\documentclass[journal]{IEEEtran}

\usepackage{amsmath, amssymb, bbm, xspace}
\usepackage{epsfig}
\usepackage{longtable}
\usepackage{color}
\usepackage{mathrsfs}

\usepackage{courier}

\newtheorem{theorem}{Theorem}[section]

\newtheorem{lemma}{Lemma}[section]

\newtheorem{proposition}[theorem]{Proposition}

\newtheorem{corollary}[theorem]{Corollary}

\newtheorem{definition}{Definition}[section]

\def\bkE{{\rm I\kern-.17em E}}
\def\bk1{{\rm 1\kern-.17em l}}
\def\bkD{{\rm I\kern-.17em D}}
\def\bkR{{\rm I\kern-.17em R}}
\def\bkP{{\rm I\kern-.17em P}}

\def\bkZ{{\bf{Z}}}

\def\bkE{{\rm I\kern-.17em E}}
\def\bk1{{\rm 1\kern-.17em l}}
\def\bkD{{\rm I\kern-.17em D}}
\def\bkR{{\rm I\kern-.17em R}}
\def\bkP{{\rm I\kern-.17em P}}

\def\bkZ{{\bf{Z}}}
\def\b12{(\beta_1,\beta_2)}

\newenvironment{proof}[1][]{{\noindent \bf Proof #1: }}{\hfill \qed \vspace{6pt}\\ }

\newcounter{example}
\renewcommand{\theexample}{\thesection.\arabic{example}}

\newcounter{remark}
\renewcommand{\theremark}{\thesection.\arabic{remark}}

\newlength{\noteWidth}
\long\def\notes#1{\ifinner
{\tiny #1}
\else
\marginpar{\parbox[t]{\noteWidth}{\raggedright\tiny #1}}
\fi\typeout{#1}}

 \def\notes#1{\typeout{read notes: #1}} 

\newcommand{\minimize}[1]{\displaystyle\minim_{#1}}
\newcommand{\minim}{\mathop{\hbox{\rm min}}}

\def\subject{\hbox{\rm s.t.}}

\def\spose#1{\hbox to 0pt{#1\hss}}

\def\text #1{\hbox{\quad#1\quad}}

\def\nthinsp{\mskip -2   mu}

\def\superstar{^{\raise 0.5pt\hbox{$\nthinsp *$}}}
\def\SUPERSTAR{^{\raise 0.5pt\hbox{$*$}}}

\def\lamstarT {\lambda^{\raise 0.5pt\hbox{$\nthinsp *$}T}}

\let\forallnew\forall
\renewcommand{\forall}{\forallnew\ }
\let\forall\forallnew

		\def\bkE{{\rm I\kern-.17em E}}
		\def\bk1{{\rm 1\kern-.17em l}}
		\def\bkD{{\rm I\kern-.17em D}}
		\def\bkR{{\rm I\kern-.17em R}}
		\def\bkP{{\rm I\kern-.17em P}}
		\def\bkY{{\bf \kern-.17em Y}}
		\def\bkZ{{\bf \kern-.17em Z}}
		\def\bkC{{\bf  \kern-.17em C}}

{\begin{list}{}%
         {\setlength{\leftmargin}{#1}}%
         \item[]%
}
{\end{list}}

		\def\bsp{\begin{split}}
		\def\beq{\begin{eqnarray}}
		\def\bal{\begin{align*}}
		\def\bc{\begin{center}}
		\def\be{\begin{enumerate}}
		\def\bi{\begin{itemize}}
		\def\bs{\begin{small}}
		\def\bS{\begin{slide}}
		\def\ec{\end{center}}
		\def\ee{\end{enumerate}}
		\def\ei{\end{itemize}}
		\def\es{\end{small}}
		\def\eS{\end{slide}}
		\def\eeq{\end{eqnarray}}
		\def\eal{\end{align*}}
		\def\esp{\end{split}}
		\def\qed{ \vrule height7.5pt width7.5pt depth0pt}  

		\def\problemsmall#1#2#3#4{\fbox
		 {\begin{tabular*}{0.47\textwidth}
			{@{}l@{\extracolsep{\fill}}l@{\extracolsep{6pt}}l@{\extracolsep{\fill}}c@{}}
				#1 & $\minimize{#2}$ & $#3$ & $ $ \\[5pt]
					  $\subject\ $ &    & $#4$ & $ $
			\end{tabular*}}
			}

	\def\cp2problem#1#2#3#4{\fbox
		 {\begin{tabular*}{0.9\textwidth}
			{@{}l@{\extracolsep{\fill}}l@{\extracolsep{6pt}}l@{\extracolsep{\fill}}c@{}}
				#1 & & $#4 $ 
			\end{tabular*}}}

		\def\bkE{{\rm I\kern-.17em E}}
		\def\bk1{{\rm 1\kern-.17em l}}
		\def\bkD{{\rm I\kern-.17em D}}
		\def\bkR{{\rm I\kern-.17em R}}
		\def\bkP{{\rm I\kern-.17em P}}
		
		\def\bkZ{{\bf{Z}}}

\newcommand {\beeq}[1]{\begin{equation}\label{#1}}
\newcommand {\eeeq}{\end{equation}}
\newcommand {\bea}{\begin{eqnarray}}
\newcommand {\eea}{\end{eqnarray}}

\def\texitem#1{\par\smallskip\noindent\hangindent 25pt
               \hbox to 25pt {\hss #1 ~}\ignorespaces}

\usepackage{cite}

\ifCLASSINFOpdf
\else
\fi
\usepackage{algorithmic}

\usepackage{array}

\begin{document}

%
\title{New Results on the Existence of Open Loop Nash Equilibria in Discrete Time Dynamic Games }
%
%
%

\author{Mathew~P.~Abraham and 
        Ankur~A.~Kulkarni
\thanks{Mathew  and Ankur are with the Systems and Control Engineering group, Indian Institute of Technology Bombay, Mumbai, India, 400076. email: \texttt{mathewp@iitb.ac.in}, \texttt{kulkarni.ankur@iitb.ac.in}}
}

\maketitle

\begin{abstract}We address the problem of finding conditions which guarantee the existence of open-loop Nash equilibria in discrete time dynamic games (DTDGs). The classical approach to DTDGs involves analyzing the problem using optimal control theory which yields results mainly limited to linear-quadratic games\cite{basar99dynamic}. We show the existence of equilibria for a class of DTDGs where the cost function of players admits a \textit{quasi-potential} function which leads to new results and, in some cases, a generalization of similar results from linear-quadratic games. Our results are obtained by introducing a new formulation for analysing DTDGs using the concept of a \textit{conjectured state} by the players. In this formulation, the state of the game is modelled as \textit{dependent} on players. Using this formulation we show that there is an \textit{optimisation problem} such that the solution of this problem gives an equilibrium of the DTDG.
    To extend the result for more general games, we modify the DTDG with an additional constraint of \textit{consistency} of the conjectured state. Any  equilibrium of the original game is also an equilibrium of this modified game with consistent conjectures.
    In the modified game, we show the existence of equilibria for DTDGs where the cost function of players admits a potential function. We end with conditions under which an equilibrium of the game with consistent conjectures is an $\epsilon$-Nash equilibria of the original game.       
    
\end{abstract} 


%
\IEEEpeerreviewmaketitle

\section{Introduction}
Many scenarios which involve more than one player with decisions to be made over finitely many stages can be modelled as Discrete Time Dynamic Games (DTDGs). In a DTDG, at each stage of the game, each player makes decisions which optimise his \textit{cost function}. A set of variables called the \textit{state} of the game evolves according to the \textit{state equation} depending on the decision of players at each stage. Many situations in power markets, robotics, network security, environmental economics, natural resource economics, industrial organisation and so on are known to come under the framework of DTDGs \cite{kannan2011game}, \cite{van2010survey}, \cite{bylka2000discrete}. 

The \textit{information structure} of the dynamic game declares what each player knows while making the decision. A commonly used information structure is the \textit{open-loop information structure} in which the players only have the information of the initial state of the game while making the decisions. This paper derives new results on the existence of \textit{Nash equilibria} under the assumption of an open-loop information structure. The resulting equilibria are termed as \textit{open-loop Nash equilibria}.

An open-loop Nash equilibrium  is a profile of strategies for which the players do not have any incentive for unilateral deviation. Guaranteeing the existence of an equilibrium has been a primary challenge in the theory of DTDGs. In general, the approach of analysing DTDGs consists of viewing each players' problem as an optimal control problem and using the results from optimal control theory \cite{basar99dynamic}. But the results that provide sufficient conditions for the existence of an equilibrium are mostly confined to linear-quadratic DTDGs \cite{engwerda2005lq}, \cite{jank2003existence}, \cite{reddy2015OLNE_DTDG}. In this paper, we introduce a new approach of analysing DTDGs which leads to new results on the existence of open-loop Nash equilibria. Furthermore, in some cases, our result generalize earlier results from linear-quadratic games.

A key step in our approach is the introduction of the concept of a \textit{conjectured state} by players for DTDGs. This allows us to model the state of the game as dependent on players. This leads to a new formulation of DTDGs with the conjectured state also considered as a decision variable. In the new formulation, we define a class of games called \textit{quasi-potential} DTDGs. In quasi-potential DTDGs, the cost function of players has a special structure. One part of the cost function admits a potential function and another part is identical for all players. Our first result gives conditions for the existence of open-loop Nash equilibria for quasi-potential DTDGs. The result is obtained by relating an equilibrium of a quasi-potential DTDG to the solution of an optimisation problem.

Following this, we modify the DTDG by introducing an additional constraint of consistency of the conjectured states\footnote{The notion of consistent conjecture used here is different from the one used by Marie and Tidball in \cite{jean2005consistent}}, which can be deciphered as follows. A player in this game is constrained to conjecture the state of the game consistently with the conjectures of other players. We will refer to this game as a \textit{DTDG with consistent conjectures}. The equilibrium that we obtain from the DTDG with consistent conjectures is a weaker notion of the equilibrium of the original game. That is, an equilibrium of the original game is also an equilibrium of the game with consistent conjectures, however, the reverse may not be true. The DTDG with consistent conjectures has a \textit{shared constraint} structure unlike the original game. Utilising the shared constraint structure, we show  the existence of equilibria for DTDGs with consistent conjectures under the assumption that the cost function of the players admits a potential function \cite{monderer96potential}. We end with a result that provides sufficient conditions for an equilibrium of the game with consistent conjectures to be an $\epsilon$-Nash equilibrium of the original game.

In addition to deriving new results on the existence of open-loop Nash equilibria, this approach has other advantages also. For example, the formulation with the conjectured state allows one to consider constraints on actions and states in a game. Also, the weaker notion of Nash equilibrium for the game with consistent conjectures can be used as a design tool since there can be some additional equilibrium points in the modified game compared to the original one.  

The rest of the paper is organised as follows. The next section provides necessary preliminaries and background for the paper. Section \ref{sec_equivalent_formulation} details the new formulation of DTDGs with the conjectured state as decision variable and its equivalence to the classical formulation.  Section \ref{sec_quasi} defines quasi-potential DTDGs and conditions for existence of an equilibrium in such games. The case of DTDGs with stage-additive cost function is also considered for quasi-potential DTDGs in this section. Section \ref{sec_modified_game} describes the DTDG with consistency condition, the shared constraint structure of this game and the existence of an equilibrium in this game with the cost of players having a potential function. This section also provides a condition for an equilibrium of the DTDG with consistent conjectures to be an $\epsilon$-Nash equilibrium of the original problem. The paper ends with a conclusion in Section \ref{conclusion}.

\section{Preliminaries \& Background} \label{preliminaries}
Before discussing the main contents of the paper, in this section, we mention some necessary definitions and results which are related to our work.  
The definition of DTDGs and the open-loop Nash equilibria in DTDGs \cite{basar99dynamic} are presented initially. This will also introduce the notation that we use throughout this paper.
\begin{definition} \label{def_classical_DTDG}
\label{DTDG_def} A DTDG with open-loop information structure consists of the following. 
\begin{enumerate}
\item An index set, $\mathcal{N}=\{1, 2,\dots,N\}$ called the \textit{players set}, where $N$ denote the number of players.

\item An index set, $\mathcal{K}=\{1, 2,\dots,K\}$ called the \textit{stage set} of the game, where $K$ is the maximum number of actions a player can make in a game.

\item A set ${U_{k}^{i}}$,  $\forall i\in \mathcal{N}$ and $\forall k\in \mathcal{K}$ called the \textit{action / control set} to which the action of player $i$ at stage $k$ belongs. The Cartesian set, $U^i \triangleq U^i_1 \times U^i_2 \times \dots \times U^i_K$ is the action set of player $i$ and $U \triangleq U^1 \times \dots \times U^N$ is the action set of the game. The action set of adversaries of player $i$ at stage $k$ and for the game is defined as
$ U^{-i}_{k} \triangleq \prod_{j \in \mathcal{N} \backslash \{i\}} U^{j}_{k}$  and
$U^{-i} \triangleq U^{-i}_1 \times U^{-i}_2 \times \dots \times U^{-i}_{K}$ respectively. 

\item A set $X_{k}$, $\forall k\in \mathcal{K} \cup \{K+1\}$ called the \textit{state space} of the game at stage $k$ to which the state of the game belongs. Since we are considering the open-loop information structure, the initial state of the game is assumed to be known to all players and is denoted by $x_1 \in X_1$.

\item A mapping $f_{k}: X_{k}\times U^{1}_{k} \times U^{2}_{k} \times \dots \times U^{N}_{k} \longrightarrow X_{k+1}$ is defined for each $k\in \mathcal{K}$ such that 
\begin{equation} \label{state_eqn}
x_{k+1}=f_{k}(x_{k},u_{k}^{1},u_{k}^{2},\dots,u_{k}^{N})
\end{equation} is the \textit{state equation} of the DTDG, where $u_{k}^{1}\in U^1_k,u_{k}^{2}\in U^2_k,\dots,u_{k}^{N}\in U^N_k$ are the actions of the players at stage $k$ and $x_k\in X_k$ is the state at stage $k$. 

\item A class of mappings denoted by $\Gamma^i_k$, $ i \in \mathcal{N}$, $ k \in \mathcal{K}$ called the \textit{strategy set} of player $i$ at stage $k$. A mapping $\gamma^i_k \in \Gamma^i_k$, given by $\gamma^i_k: X_1 \rightarrow U^i_k$ is the \textit{strategy} of player $i$ at stage $k$. The aggregate mapping $\gamma^i=\{\gamma^i_1,\gamma^i_2,\dots,\gamma^i_K\} \in \Gamma^i = \prod\limits_{k=1}^{K}\Gamma^i_k$ is the strategy of player $i$ in the game.

\item A function $ J^{i}:U \times X_2 \times \dots\times X_{K+1}  \longrightarrow \mathbb{R}$ is defined for each $i \in \mathcal{N}$ called the \textit{cost function of Player $i$} in the game of $K$ stages.
\end{enumerate}

For each fixed initial state $x_1$ and for fixed $N$-tuple of permissible strategies $\{\gamma^i \in \Gamma^i;i \in \mathcal{N}\}$, there exist a unique set of vectors $\{u^i_k \equiv \gamma^i_k(x_1); i \in \mathcal{N}, k \in \mathcal{K}\}$ and the state evolves according to the state equation \eqref{state_eqn}. Substitution of these quantities into $J^i,i \in \mathcal{N}$ leads to a unique $N$-tuple of numbers reflecting corresponding costs to the players. This implies the existence of $
L^{i}:\Gamma^1 \times \Gamma^2 \times \dots \times\Gamma^{N} \longrightarrow \mathbb{R}
$, the cost function in the strategy space. 
That is
$L^i (\gamma^{1},\dots,\gamma^{N})  \equiv J^i (u^1,\dots,u^N, x_1,x_2,\dots,x_{K+1}),$
where $\gamma^i=\{\gamma^i_1,\gamma^i_2,\dots,\gamma^i_K \} \in \Gamma^i; i \in \mathcal{N}$ and $u^i=(u^i_1, u^i_2,\dots, u^i_K) \in U^i; i \in \mathcal{N}$ and each $u^i_k= \gamma^i_k(x_1)$ and $x_{k+1}=f_{k}(x_{k},u_{k}^{1},u_{k}^{2},\dots,u_{k}^{N}), k \in \mathcal{K}$.

\end{definition}

 In a DTDG with open-loop information structure, a player's problem is to decide his strategies $\gamma^i \in \Gamma^i$ which minimises his cost function $L^{i}$. Often the cost function is expressed as a stage-wise additive function which is defined as follows.
 
 \begin{definition} \label{def_stage_additive}
     In an $N$-person DTDG of $K$-stages, player $i$'s cost function is said to be \textit{stage-additive} if there exist functions $g^i_k: X_{k+1} \times U^1_k \times \dots \times U^N_k \times X_k \rightarrow \mathbb{R}, \forall k \in \mathcal{K}$ such that, 
     \begin{align*}
         J^i(u^1,\dots,u^N,x_{1},\dots,x_{K+1}) =  \sum\limits_{k=1}^{K} g^i_k(x_{k+1},u^1_k,\dots,u^N_k,x_k),
     \end{align*}
 where $u^i=(u^i_1,\dots,u^i_K) \in U^i, \forall i \in \mathcal{N}$ and $x_{k} \in X_{k}, \forall k \in \mathcal{K} \cup \{K+1\}$.
 \end{definition}

We define the open-loop Nash equilibrium in a DTDG as follows \cite{basar99dynamic}.
\begin{definition}
 An $N$-tuple of strategies $\{\gamma^{i*}\in\Gamma^i;i\in \mathcal{N}\}$ constitute an \textit{open-loop Nash equilibrium} if and only if the following inequalities are satisfied  $\forall i \in \mathcal{N}$.
\begin{multline*}
L^i(\gamma^{1*},\gamma^{2*},\dots,\gamma^{i*},\dots,\gamma^{N*}) \leq  L^i(\gamma^{1*},\dots,\gamma^{(i-1)*},\\ \gamma^{i},\gamma^{(i+1)*},\dots,\gamma^{N*}), \forall \gamma^i \in \Gamma^i.
\end{multline*}

\end{definition}

Let $\{u^{i*}\in U^i ; i\in \mathcal{N}\}$ be the actions of players corresponding to the equilibrium strategies $(\gamma^{1*},\gamma^{2*},\dots,\gamma^{N*})$. Unlike other information structures, in the open-loop information structure, the strategy $\gamma^i_k \in \Gamma^i_k, i\in \mathcal{N}, k \in \mathcal{K}$ is a constant function of $x_1$. Hence we can write that the $N$-tuple of actions $\{u^{i*} \in U^i ; i\in \mathcal{N}\}$ corresponding to the equilibrium strategies constitute an \textit{open-loop Nash equilibrium} if and only if the following inequalities are satisfied $\forall i \in \mathcal{N}$.
\begin{align*} \label{NE_eqn}
J^i(u^{1*},\dots,u^{i*},\dots,u^{N*},x_{1},x_2^*,\dots,x^*_{K+1}) \leq  J^i(u^{1*},\dots,\\ u^{(i-1)*},u^{i},u^{(i+1)*}, \dots,u^{N*},x_1,x_2,\dots,x_{K+1}), \forall u^i \in U^i,
\end{align*}
where $x_{k+1}^*=f_{k}(x_{k}^*,u_{k}^{1*},u_{k}^{2*},\dots,u_{k}^{N*}), k \in \mathcal{K}$ and $x_{k+1}=f_{k}(x_{k},u_{k}^{1*},\dots,u^{(i-1)*}_k,u^{i}_k,u^{(i+1)*}_k,\dots, u_{k}^{N*}), k \in \mathcal{K}$.

A commonly used line of analysis for DTDGs is given as follows. A necessary condition for the existence of an open-loop Nash equilibrium based on Pontryagin's principle can be derived as in Ba\c{s}ar and Olsder \cite{basar99dynamic}. A sufficient condition for linear-quadratic DTDGs can also be obtained by solving a recursive Riccati equation. Under some additional assumptions, it is shown in \cite{jank2003existence} that the two-point boundary value problem is uniquely solvable for a linear-quadratic case. A recent paper \cite{reddy2015OLNE_DTDG} by Reddy and Zaccour considers a class of linear-quadratic DTDGs with linear constraints on control and state. 
Another potential line of analysis is as follows. Since an open-loop Nash equilibrium is equivalently given by a sequence of actions $u^{i*} \in U^i, i \in \mathcal{N}$ (rather than strategies $\gamma^{i*} \in \Gamma^i, i \in \mathcal{N}$), the open-loop Nash equilibrium is also the equilibrium of a \textit{static}  game obtained by substituting the state equation~\eqref{state_eqn} into the cost function of each player. Evidently, this line of analysis succeeds only if the cost function and the dynamics take simple forms. In this paper, we introduce a different approach by reformulating DTDGs with the concept of a conjectured state as shown in the next section.

\section{Reformulation of DTDGs}
\label{sec_equivalent_formulation}

An important step in deriving the new results on the existence of equilibria is the reformulation of DTDGs with the notion of the \textit{conjectured state}. In the classical definition of DTDGs, the state of the game is defined as independent of players and it evolves ``in the background'' according to the state equation depending on the actions of players. For a player, a conceivable manner of playing in a DTDG is to  make an independent conjecture about the state trajectory. Motivated by this, we model the state of the game as dependent on players. We will refer to this new formulation as \textit{state conjecture formulation}. The new formulation has the advantage of easier analysis than the classical formulation. Furthermore, in this section, we show the equivalence of the state conjecture formulation to the classical one. A notion of conjecture is also used in \cite{jean2005consistent}. But the conjecture used there is about the rivals actions using the state or control information which is different from our notion of the state conjecture. 

 In the state conjecture formulation, the state space of player $i$ at stage $k$ is defined as set $X_{k+1}^{i}, \forall i\in \mathcal{N}, \forall k\in \mathcal{K}$. In order to make the new formulation consistent with the original, the state space is defined such that $X^i_{k+1} \equiv X_{k+1}, \forall i \in \mathcal{N}, \forall k \in \mathcal{K}$. The state space of player $i$ over stages is given by
$X^i \triangleq \prod\limits_{k=2}^{K+1}X^i_k $ and the state space of the game as $ X \triangleq \prod\limits_{i=1}^{N}X^i$. We define the state space of adversaries of player $i$ at stage $k$ as 
$ X^{-i}_{k} \triangleq \prod_{j \in \mathcal{N} \backslash \{i\}} X^{j}_{k}$ and the state space of adversaries of player $i$ for the game as $X^{-i} \triangleq \prod\limits_{k=2}^{K+1}X^{-i}_k$.
The dependence of the state space on players leads to respective changes in the state equation, the strategies and the cost function. 
The state mapping $f_{k}: X^{i}_{k}\times U^{1}_{k} \times U^{2}_{k} \times ... \times U^{N}_{k} \longrightarrow X^{i}_{k+1}$ is defined for each $k\in \mathcal{K}$ so that, 
\begin{equation} \label{state_eqn2}
x_{k+1}^{i}=f_{k}(x_{k}^{i},u_{k}^{1},u_{k}^{2},\dots,u_{k}^{N})
\end{equation} is the {state equation} of the DTDG in state conjecture formulation.
The permissible strategies in new formulation become $\gamma^i_k:X_1 \rightarrow U^i_k,  i \in \mathcal{N},  k \in \mathcal{K}$ 
and the cost function of player $i$ is a mapping given by
$J^{i}:U \times X_1 \times X^{i}_{2} \times \dots \times X^{i}_{K+1} \longrightarrow \mathbb{R}$ which is defined for each $i\in \mathcal{N}$.

\subsubsection{Player $i$'s problem} \label{playeri_in_newformulation}
Consider a player's problem in the state conjecture formulation of a DTDG. Let $i \in \mathcal{N}$ be a player in the game. Given the rivals actions $u^{-i} = (u^{-i}_1,\dots,u^{-i}_K) \in U^{-i}$  and the initial state $x_1$, the player $i$'s problem is denoted by $P_i (u^{-i};x_1$) and is given by the following optimisation problem.
$$
\problemsmall{$P_i (u^{-i};x_1$)}
{u^i,x^i}
{J^i(u^i,x^i;u^{-i},x_1)}
{(u^i,x^i) \in \Omega_i (u^{-i};x_1),}
$$
where,
\begin{align*}
{\Omega_{i}}({u^{-i};x_1}) &= 
\{
\hat{u}^{i}_1,\dots,\hat{u}^{i}_{K}, \hat{x}^{i}_{2}, \dots \hat{x}^{i}_{K+1}| \hat{u}^{i}_{k}\in U^{i}_{k},\forall k \in \mathcal{K}, \\ &\hat{x}^i_{k+1} \in X^i_{k+1}, \forall k \in \mathcal{K}, \hat{x}^{i}_{2} = f_{1}(x_{1}, \hat{u}^{i}_{1},u^{-i}_{1}),   \\&\hat{x}^{i}_{k+1} = f_{k}(\hat{x}^{i}_{k}, \hat{u}^{i}_{k}, u^{-i}_{k}), \forall k \in \mathcal{K}\backslash \{1\}
\}.
\end{align*}
Since the state trajectory of the game $x^i = (x^i_2,\dots,x^i_{K+1}) \in X^i$ is dependent on players, it is also considered as a decision variable as opposed to the classical formulation.  
Define the Cartesian product of the feasible set $\Omega_{i}({u^{-i};x_1})$ as  
$\Omega(u,x) \triangleq\prod_{i \in \mathcal{N}}\Omega_i (u^{-i};x_1)$, where $u=(u^1,u^2,\dots,u^N) \in U$ and $x=(x^1,x^2,\dots,x^N) \in X$. The fixed points of the set value map $\Omega(u,x)$ is denoted by $\mathcal{F}$ which is given by
$\mathcal{F} \triangleq \{(u,x)| (u,x)\in \Omega(u,x) \}.$ We can write the set $\mathcal{F}$ in the expanded form as follows. 
\begin{align*}
\mathcal{F}=\{(u,x)| u\in U, x^i \in X^i, \forall i \in \mathcal{N}, x^i_2=f_1(x_1,u^1_1,\dots, u^N_1), \\ \forall i \in \mathcal{N}, x^i_{k+1}=f_k(x^i_k,u^1_k,\dots,u^N_k), \forall i \in \mathcal{N}, \forall k \in \mathcal{K}\}.
\end{align*}
 Notice that the initial state $x_1$ is a parameter on which the set $\mathcal{F}$ depends. So in the set $\mathcal{F}$, for a fixed $x_1$, the state trajectories of players are consistent, i.e., for all $i \in \mathcal{N}$, $x^i=x^j, \forall j \in \mathcal{N}\backslash\{i\}$. Using the state conjecture formulation of DTDG, we redefine the open-loop Nash equilibrium as follows.

\begin{definition}
    A tuple of strategies $\{(u^{i*},x^{i*});i\in \mathcal{N}\} \in \mathcal{F}$ is an open-loop Nash equilibrium of DTDG if $\forall i \in \mathcal{N}$, given $u^{-i*} \in U^{-i}$, 
   \begin{multline*}
    J^{i}(u^{i*},x^{i*};u^{-i*},x_1) \leq
     J^{i}({u}^i,{x}^i;u^{-i*},x_1),
   \\ \forall ({u}^i,{x}^i) \in \Omega_i (u^{-i*};x_1).
     \end{multline*}
\end{definition}
\subsubsection{Equivalence to the classical formulation} 
The equivalence of the classical and the state conjecture formulation in terms of the open-loop Nash equilibria is given in the following proposition.
\begin{proposition}
Consider a DTDG as defined in Definition \ref{def_classical_DTDG}. Suppose the state space is equivalent in the state conjecture and the classical formulations ($X^i_{k+1} \equiv X_{k+1}, \forall i \in \mathcal{N}, k\in \mathcal{K}$). For a given initial state $x_1 \in X_1$, then $(u^*,x^*) \in \mathcal{F}$ is an open-loop Nash equilibrium of a DTDG in the state conjecture formulation if and only if $u^* \in U$ is an open-loop Nash equilibrium in the classical formulation with the state trajectory at equilibrium given by any component $x^{i*} \in X^i$ of $x^{*}$.    
\end{proposition}    

\begin{proof}
Let $(u^*,x^*) \in \mathcal{F}$ be an equilibrium of DTDG in state conjecture formulation. Consider the problem of player $i$ in state conjecture formulation given by $P_i (u^{-i*};x_1)$, where $u^{-i*}$ is the equilibrium actions of other players and $x_1$, the initial state. The player $i$'s problem in the state conjecture formulation with state expanded over stages is denoted by problem $\mathcal{P}_i (u^{-i*};x_1)$ which is given as follows. 
$$
     \problemsmall{$\mathcal{P}_i (u^{-i*};x_1)$}
     {u^i,x^i}
     {J^i(u^i,x^i_2,\dots,x^i_{K+1};u^{-i*},x_1)}
                  {(u^i,x^i_2,\dots,x^i_{K+1}) \in \Omega_i (u^{-i*};x_1).}
     $$
By substituting the state equation~\eqref{state_eqn2} for the states in the cost function from the feasible set, we redefine the problem in action space as given below. 
$$
     \problemsmall{}
     {u^i,x^i}
     {J^i(u^i,f_1(x_1,u^i_1,u^{-i*}_1),\dots;u^{-i*},x_1)}
                  {u^i_k \in U^i_k, \forall k \in \mathcal{K}. }
     $$ 
This is indeed the classical DTDG problem of player $i$ in the action space. Hence, any equilibrium in the state conjecture formulation is an equilibrium in the classical formulation. Conversely, for any equilibrium in the space of actions, player $i$'s problem can be written as $\mathcal{P}_i (u^{-i*};x_1)$ and hence an equilibrium in the state conjecture formulation.
Hence $(u^*,x^*) \in \mathcal{F}$ is an equilibrium in state conjecture formulation if and only if $u^* \in U$ is an equilibrium in the original formulation.
\end{proof}
Thus, we have a new formulation equivalent to the classical one. We will exploit this new formulation to derive our existence results.

\section{QUASI-POTENTIAL DTDG} \label{sec_quasi}
Our first result on the existence of open-loop Nash equilibria for DTDGs is presented in this section. The result provides conditions for the  existence of equilibria for a class of games called quasi-potential DTDGs.  The idea is to split the cost function of players into two parts, one part that admits a potential function and another part which is identical for all players.  A quasi-potential game is a special case of a potential game \cite{monderer96potential}, introduced by Kulkarni et. al. in \cite{kulkarni2015existence} for multi-leader multi-follower games. Here, we extend the concept of quasi-potential to DTDGs. The definition of quasi-potential DTDGs is given as follows.

\begin{definition} 
\label{quasi_def}
Consider a DTDG defined in state conjecture formulation with the cost function of players given by $\{J^1, J^2,\dots, J^N\}$. The DTDG is said to be \textit{quasi-potential} if the cost function of players admits a quasi-potential function for which the following has to hold.
\begin{enumerate}
\item There exist functions $\Phi_1, \Phi_2,\dots,\Phi_N$ and $h$ such that for all $i \in \mathcal{N}$, player $i$'s cost function can be written as $J^i(u,x_1,x^i) \equiv \Phi_i(u,x_1)+h(u,x_1,x^i), \forall u \in U, \forall x^i \in X^i,$ 

\item and there exist a function $\pi$ such that for all $i \in \mathcal{N}$ and for all ${u}^{-i} \in U^{-i}$, we have $\Phi_i(\tilde{u}^i;u^{-i},x_1)-\Phi_i(\hat{u}^{i};u^{-i},x_1)= \pi(\tilde{u}^i;u^{-i},x_1)-\pi(\hat{u}^{i};u^{-i},x_1), \forall  \tilde{u}^i,\hat{u}^{i} \in U^i$.
\end{enumerate}
\end{definition}
The function $\pi+h$ is termed as a \textit{quasi-potential function} of the quasi-potential DTDG. Note that the function $h$ is independent of players, but the argument of the function is dependent on players. We utilize this structure of quasi-potential DTDGs to provide the existence of equilibria  in such games. 

\subsection{Existence of equilibria in quasi-potential DTDGs}\label{quasi_existence}
We need to establish a lemma in order to prove the existence of open-loop Nash equilibria in quasi-potential DTDGs. For that, consider a set $\mathcal{F}^q$ defined as follows.
\begin{align*}
    \mathcal{F}^q &\triangleq \{(u^1_{1:K},\dots,u^N_{1:K},w_2,w_3,\dots,w_K,w_{K+1})| u^1_{1:K},\dots,\\& u^N_{1:K} \in U, w_{k+1} \in X_{k+1}, \forall k \in \mathcal{K}, w_2=f_1(x_1,u^1_1, \dots,u^N_1),\\& w_{k+1}=f_k(w_k,u^1_k,\dots,u^N_k), \forall k \in \mathcal{K} \backslash \{1\} \}.
\end{align*}
Note that the initial state $x_1$ is a parameter on which the set $\mathcal{F}^q$ depends. Now we consider the following lemma.
\begin{lemma}\label{prop1}
   Consider a DTDG with initial state $x_1$. For some $i \in \mathcal{N}$, given $u^{-i} \in U^{-i}$, a point $(u^i,w_2,w_3,\dots,w_{K+1})$ is feasible for the problem given by $P_i (u^{-i}; x_1$) if and only if $(u,w_2,w_3,\dots,w_{K+1}) \in \mathcal{F}^q$. That means,
    \begin{multline*}
        (u^i,w_2,w_3,\dots,w_{K+1})\in \Omega_i (u^{-i};x_1) \iff \\(u,w_2,w_3,\dots,w_{K+1}) \in \mathcal{F}^q. 
    \end{multline*}
\end{lemma}

\begin{proof}
    ``$\Rightarrow$"
    For given $u^{-i} \in U^{-i}$ and initial state $x_1$, consider a point $(u^i,w_2,w_3,\dots,w_{K+1})\in \Omega_i (u^{-i};x_1)$. That means, $u^i \in U^i, w_2=f_1(x_1,u^i_1,u^{-i}_1),  w_{k+1}=f_k(w_k,u^i_k,u^{-i}_k), \forall k \in \mathcal{K} \backslash \{1\}$. Now by combining $u^i \in U^i$ and $u^{-i} \in U^{-i}$ we can rewrite the above equations as $u \in U, w_2=f_1(x_1,u^i_1,u^{-i}_1),w_{k+1}=f_k(w_k,u^i_k,u^{-i}_k),\forall k \in \mathcal{K} \backslash \{1\}$. Hence, $(u,w_2,w_3,\dots,w_{K+1}) \in \mathcal{F}^q$.
    
    ``$\Leftarrow$" Suppose $(u,w_2,w_3,\dots,w_{K+1}) \in \mathcal{F}^q$. Hence, $u \in U, w_2=f_1(x_1,u^i_1,u^{-i}_1), w_{k+1}=f_k(w_k,u^i_k,u^{-i}_k), \forall k \in \mathcal{K} \backslash \{1\}$. For some $i\in \mathcal{N}$, we separate the actions of player $i$  and adversaries as $u^i \in U^i$ and $u^{-i} \in U^{-i}$. For some fixed  $u^{-i}\in U^{-i}$ and $x_1$, the equations can be rewritten as $u^i \in U^i, w_2=f_1(x_1,u^i_1,u^{-i}_1), w_{k+1}=f_k(w_k,u^i_k,u^{-i}_k),\forall k \in \mathcal{K} \backslash \{1\}$. Hence, $(u^i,w_2,w_3,\dots,w_{K+1})\in \Omega_i (u^{-i};x_1)$.
\end{proof}
For a quasi-potential DTDG with quasi-potential function $\pi + h$, consider an optimisation problem denoted by $P_q$. 
    $$
    \problemsmall{$P_q$}
    {({u},w_2,\dots,w_{K+1})}
    {\pi({u},x_1)+h({u},x_1,w_2,w_3,\dots,w_{K+1})}
                 {({u},w_2,w_3,\dots,w_{K+1}) \in \mathcal{F}^{q} }
    $$
Note that the objective function of the problem $P_q$ is the quasi-potential function $\pi+h$ and the feasible set is $\mathcal{F}^q$. We relate the solution of $P_q$ to an equilibrium of a quasi-potential DTDG which is given in the following theorem. 
\begin{theorem}
\label{minNE_quasi}
Consider a quasi-potential DTDG with a quasi-potential function $\pi+h$. Consider the optimisation problem given by $P_q$ for the quasi-potential DTDG. If $({u},w_2,w_3,\dots,w_{K+1})\in \mathcal{F}^{q}$ is a global minimizer of the problem $P_q$, then $({u},x) \in \mathcal{F}$ is an open-loop Nash equilibrium of the quasi-potential DTDG with $x^i_{k+1}=w_{k+1},\forall k\in \mathcal{K}, \forall i\in \mathcal{N}.$ 
\end{theorem}

\begin{proof}
Let $({u},w_2,w_3,\dots,w_{K+1}) \in \mathcal{F}^{q}$ be the minimizer of the problem $P_q$. Then,
\begin{align*}
{\pi({u},x_1)+h({u},x_1,w_2,w_3,\dots,w_{K+1})} \leq   {\pi(\tilde{u}),x_1}+  {h(\tilde{u},x_1,}\\{\tilde{w}_2,  \dots,\tilde{w}_{K+1})}, \forall (\tilde{u},\tilde{w}_2,\tilde{w}_3,\dots,\tilde{w}_{K+1}) \in \mathcal{F}^{q}.
\end{align*}
By splitting the actions $u, \tilde{u} \in U$ over player $i$ as $u=(u^i,u^{-i})$ and $\tilde{u}=(\tilde{u}^{i},\tilde{u}^{-i})$ respectively, we can rewrite as,
\begin{multline*}
{\pi(u^i,u^{-i},x_1)+h(u^i,u^{-i},x_1,w_2,w_3,\dots,w_{K+1})} \leq  \\ {\pi(\tilde{u}^{i},\tilde{u}^{-i},x_1)+h(\tilde{u}^{i},x_1,\tilde{u}^{-i},\tilde{w}_2,\tilde{w}_3,\dots,\tilde{w}_{K+1})}, \\ \forall (\tilde{u}^{i},\tilde{u}^{-i},\tilde{w}_2,\tilde{w}_3,\dots,\tilde{w}_{K+1}) \in \mathcal{F}^{q}.
\end{multline*}
The inequality still holds even if we replace $\tilde{u}^{-i} \in U^{-i}$ by $u^{-i} \in U^{-i}$. Hence, the above inequality can be rewritten as,
\begin{multline*}
{\pi({u^i,u^{-i},x_1})+h({u^i,u^{-i}},x_1,w_2,w_3,\dots,w_{K+1})} \leq \\ {\pi(\tilde{u}^{i},{u}^{-i},x_1)+ h(\tilde{u}^{i},{u}^{-i},,x_1\tilde{w}_2,\tilde{w}_3,\dots,\tilde{w}_{K+1})}, \\ \forall (\tilde{u}^{i},{u}^{-i},\tilde{w}_2,\tilde{w}_3,\dots,\tilde{w}_{K+1}) \in \mathcal{F}^{q}. 
\end{multline*}
Using Lemma \ref{prop1}, the above inequality can be rewritten as,
\begin{multline*}
{\pi({u^i,u^{-i},x_1})+h({u^i,u^{-i}},x_1,w_2,w_3,\dots,w_{K+1})} \leq \\ {\pi(\tilde{u}^{i},u^{-i},x_1)+h(\tilde{u}^{i},u^{-i},x_1,\tilde{w}_2,\tilde{w}_3,\dots,\tilde{w}_{K+1})},\\ \forall (\tilde{u}^{i},\tilde{w}_2,\tilde{w}_3,\dots,\tilde{w}_{K+1}) \in \Omega_{i}(u^{-i};x_1), \forall i \in \mathcal{N}.
\end{multline*}
Since the game is a quasi-potential DTDG and $(w_2,w_3,\dots,w_{K+1})=x^{i} \in X^i, \forall i \in \mathcal{N}$,
\begin{multline*}
{\Phi_i({u^i,u^{-i},x_1})+h({u^i,u^{-i}},x_1,x^{i})} \leq \\ {\Phi_i(\tilde{u}^{i},u^{-i},x_1)+h(\tilde{u}^{i},u^{-i},x_1,\tilde{w}_2,\tilde{w}_3,\dots,\tilde{w}_{K+1})},\\ \forall (\tilde{u}^{i},\tilde{w}_2,\tilde{w}_3,\dots,\tilde{w}_{K+1}) \in \Omega_{i}(u^{-i};x_1), \forall i \in \mathcal{N}.
\end{multline*}
Hence, $ \forall i \in \mathcal{N}$,
\begin{multline*}
J^{i}({u},x^i,x_1) \leq  J^{i}(\tilde{u}^i ,x_1,\tilde{w}_2,\tilde{w}_3,\dots,\tilde{w}_{K+1};u^{-i}), \\ \forall (\tilde{u}^{i},\tilde{w}_2,\tilde{w}_3,\dots,\tilde{w}_{K+1}) \in \Omega_i (u^{-i};x_1).
\end{multline*}
\end{proof}
Thus, if the optimisation problem $P_q$ has a minimizer then there exists an open-loop Nash equilibrium for the quasi-potential DTDG.
The following results gives conditions which guarantee the existence of a minimum for the problem $P_q$.
\begin{corollary}\label{Coro_with_compactness}
A quasi-potential DTDG admits an open-loop Nash equilibrium if the quasi-potential function, ${\pi+h}$ is continuous and the set $\mathcal{F}^q$ is non empty and compact. The set  $\mathcal{F}^q$ is compact if the sets $U,X$ are compact and $f_k$ is continuous for all $k \in \mathcal{K}$.
\end{corollary}
Corollary \ref{Coro_with_compactness} can be easily proved using {Weierstrass Theorem}.
But the requirement of compactness of state space in Corollary \ref{Coro_with_compactness} may be restrictive in some cases. The following result gives the existence of equilibria without the condition of compactness of state or action spaces, instead imposing the condition of coercivity of quasi-potential function. The coercivity definition is given as follows.

\begin{definition}
A quasi-potential function  ${\pi+h}$ is coercive if \\ $\underset{(u,w_2,\dots,w_{K+1})\in \mathcal{F}^q }{\underset{||u,w_2,\dots,w_{K+1}||\rightarrow\infty}{\liminf}}  {\pi({u},x_1)+h({u},x_1,w_2,\dots,w_{K+1})} = \infty.$
\end{definition}

\begin{corollary}\label{Cor:quasi_coercive}
A quasi-potential DTDG admits an open-loop Nash equilibrium if the quasi-potential function, ${\pi+h}$ is continuous and coercive and $f_k$ is continuous for each $k \in \mathcal{K}$.
\end{corollary}

Since the state conjecture formulation is equivalent to the original formulation, an equilibrium of a DTDG in this new formulation with cost function of players admitting a quasi-potential function is indeed an equilibrium in the original formulation. 
\subsection{Stage-additive quasi-potential DTDG}
Now we consider the case of stage-additive DTDG as defined in Definition \ref{def_stage_additive}.
For such games the stage-wise cost functions $g^1_k,g^2_k,\dots,g^N_k$ admits a quasi-potential function if the following conditions are satisfied for each $k \in \mathcal{K}$.

\begin{enumerate}
    \item There exist functions $\Psi^1_k,\Psi^2_k,\dots,\Psi^N_k$ and $h_k$ such that $\forall i \in \mathcal{N}, g^i_k$ is given by, 
    \begin{align*}
        g^i_k(x^i_{k+1},u^1_k,\dots,u^N_k,x^i_k) \equiv  \Psi^i_k(u^1_k,\dots,u^N_k)+h_k(x^i_k,\\u^1_k,\dots,u^N_k,x^i_{k+1}), \forall u^i_k \in U^i_k, \forall x^i_{k+1}\in X^{i}_{k+1},\forall k \in \mathcal{K},
    \end{align*}  
    \item and there exist functions $\pi_k, \forall k \in \mathcal{K}$ such that for all $i \in \mathcal{N}$ and for all $ u^{-i}_k \in U^{-i}_k$,
    \begin{multline*}
        \Psi^i_k(\tilde{u}^i_k,u^{-i}_k) - \Psi^i_k(\hat{u}^i_k,u^{-i}_k) = \pi_k(\tilde{u}^i_k,u^{-i}_k) -\\ \pi_k (\hat{u}^i_k,u^{-i}_k), \forall  \tilde{u}^i_k, \hat{u}^i_k \in U^i_k.
    \end{multline*}
\end{enumerate}
The function $\pi_k + h_k$ is a \textit{stage-wise quasi-potential function} for the \textit{stage-wise quasi-potential DTDG}. The following proposition shows that if the stage-wise cost function of players admit a stage-wise quasi-potential function, then the overall cost functions also admit a quasi-potential function. 
\begin{proposition}
    Consider a stage-additive DTDG defined in Definition \ref{def_stage_additive}. If the stage-wise cost functions $g^i_k, i \in \mathcal{N}$ admit a stage-wise quasi-potential function, then the stage-additive cost functions $J^1,J^2,\dots,J^N$ also admit a quasi-potential function.
\end{proposition}

\begin{proof}
%
    The stage-additive cost function of player $i$ in a stage-wise quasi-potential DTDG is given by,
    \begin{multline*}
        J^i(u^1,u^2,\dots,u^N,x_1,x^i_2,\dots,x^i_{K+1}) = \\ \sum\limits_{k=1}^{K} \Psi^i_k(u^1_k,\dots,u^N_k)+ \sum\limits_{k=1}^{K}h_k(x^i_k,u^1_k,\dots,u^N_k,x^i_{k+1}),
    \end{multline*}
    where $u^i_k \in U^i_k,u^i\in U^i,x^i_{k+1}\in X^i_{k+1}, \forall i \in \mathcal{N}$.
The first sum admits a potential function given by $\sum\limits_{k=1}^{K} \pi_k$. The second sum is independent of players though the argument depends on. Hence, the stage-additive cost function of players admit a quasi-potential function.
\end{proof}

For a stage-wise quasi-potential DTDG with stage-wise quasi-potential function $\pi_k + h_k$, consider the following optimisation problem denoted by $P_{q}^{s}$ with $w_1=x_1$.
$$
\problemsmall{$P_{q}^{s}$}
{({u},w_2,\dots,w_{K+1})}
{\sum\limits_{k=1}^{K}\Bigl[\pi_k({u^1_k, \dots, u^N_k})+h_k({u},w_k,w_{k+1})\Bigr]}
{({u},w_2,w_3,\dots,w_{K+1}) \in \mathcal{F}^{q}. }
$$
It can be seen from Theorem \ref{minNE_quasi} that the minimizer of the problem $P_{q}^{s}$ is an open-loop Nash equilibrium of the stage-wise quasi-potential DTDG with $x^i_{k+1}=w_{k+1}, \forall i\in \mathcal{N}, \forall k\in \mathcal{K}$. The problem $P_{q}^{s}$ has the structure of a  standard control problem for which a vast theory is available which can be utilized to determine the equilibria.

It can be seen that the linear-quadratic DTDGs under certain assumptions  comes under the class of stage-wise quasi-potential DTDGs. We first provide the definition of linear-quadratic DTDGs.
\begin{definition}\label{def_LQ}
    An $N$-person DTDG is of linear-quadratic type if $U^i_k= \mathbb{R}^m_i(i\in \mathcal{N}, k \in \mathcal{K})$, and $$
    f_k(x^i_k,u^1_k,\dots,u^N_k)=A_kx^i_k+\sum\limits_{j\in\mathcal{N}}b^j_ku^j_k,$$ 
     \begin{align*}
         g^i_k(x^i_{k+1},u^1_k,\dots,u^N_k,x^i_k)=  \frac{1}{2}\Bigl( x^{i'}_{k+1}Q^i_{k+1}&x^i_{k+1}+ \\& \sum\limits_{j \in \mathcal{N}}u^{j'}_kR^{ij}_ku^j_k\Bigr),
    \end{align*}
    where $A_k,B_k,Q^i_k,R^{ij}_k$ are matrices of appropriate dimensions, $Q^i_{k+1}$ is symmetric, $R^{ii}_k \succ 0$ and  $u^i_k \in U^i_k,x^i_{k+1}\in X^{i}_{k+1},\forall i \in \mathcal{N},\forall k \in \mathcal{K}$.
    \end{definition}

\begin{lemma}\label{Lemma:quasi_LQ}
    For a linear-quadratic DTDG with $Q_{k+1}\triangleq Q^i_{k+1}=Q^j_{k+1}, \forall i,j \in \mathcal{N}, \forall k \in \mathcal{K}$, the stage-wise cost function of players, $g^i_k$ admit a quasi-potential function.
    \end{lemma}
    
    \begin{proof}
For a linear-quadratic DTDG, the stage-wise cost function is given by, 
\begin{align*}
     g^i_k(x^i_{k+1},u^1_k,\dots,u^N_k,x^i_k)= \frac{1}{2}\Bigl( x^{i'}_{k+1}Q_{k+1}&x^i_{k+1}+ \\& \sum\limits_{j \in \mathcal{N}}u^{j'}_kR^{ij}_ku^j_k\Bigr).
\end{align*}
It can be easily verified that the second summation term is a quadratic function which admits a potential function with stage-wise potential function $\pi_k (u^1_k,\dots,u^N_k) = \frac{1}{2} \sum\limits_{i=1}^{N} u^{i'}_k R^{ii}_{k} u^{i}_k$. Hence, in the definition of quasi-potential DTDG, 
$\Psi^i_k = \frac{1}{2} \sum\limits_{j \in \mathcal{N}}u^{j'}_kR^{ij}_ku^j_k $. Also, the first function is identical for all players and the argument depends on the state conjectured by players.
Hence, $h_k = \frac{1}{2} x^{i'}_{k+1}Q^i_{k+1}x^i_{k+1}$, which implies that $g^i_k$ admits a quasi-potential function.    
        \end{proof}
        The fact that the stage-wise cost function of the linear-quadratic DTDGs under certain assumptions admit a stage-wise quasi-potential function can be utilized to provide the existence of equilibria in such cases which are given by the following corollary. 
        
        \begin{corollary}
        Consider a linear-quadratic DTDG as defined in \ref{def_LQ}. Suppose the assumptions of Lemma \ref{Lemma:quasi_LQ} hold with $Q_{k+1} \succ 0, \forall k \in \mathcal{K}$. Then there exists an equilibrium for the linear-quadratic DTDG which is given by the minimizer of the problem $P_q^s$. 
            \end{corollary}
        \begin{proof}
        Since $Q_{k+1} \succ 0, \forall k \in \mathcal{K}$ and $R^{ii}_k \succ 0, \forall i \in \mathcal{N}, \forall k \in \mathcal{K}$, the quasi-potential function $\sum\limits_{k\in\mathcal{K}}(\pi_k + h_k$) as defined in the proof of Lemma \ref{Lemma:quasi_LQ} for the linear-quadratic DTDG is coercive. Then the result follows directly from Corollary \ref{Cor:quasi_coercive}.
        \end{proof}
        Thus, using the state conjecture formulation, the existence of open-loop Nash equilibria can be guaranteed for games with cost function of players admitting quasi-potential function which also includes a class linear-quadratic DTDGs. In the next section, we introduce some additional constraints on the state conjecture formulation which can provide new existence results for DTDGs.

\section{DTDG with consistent conjectures}\label{sec_modified_game}
In the state conjecture formulation of DTDGs, each player conjectures the state of the game independently. In this section, we modify the DTDG with an additional constraint of \textit{consistency of the conjectured states}. This constraint requires that each player conjecture the state of the game consistently with other players. We will refer to this game as a \textit{DTDG with consistent conjectures}. Even though the DTDG with consistent conjectures is different from the original DTDG, an equilibrium of the original game is also an equilibrium of the game with consistent conjectures. The DTDG with consistent conjectures has a special structure called \textit{shared constraint} structure which the original game does not have. We utilize this structure for giving conditions for the existence of equilibria for DTDGs when the cost functions of players admit a potential function. 
In this section, we also discuss conditions under which an equilibrium of the DTDG with consistent conjectures is an $\epsilon$-Nash equilibrium of the original game.

A player's problem in the game with consistent conjectures is given in the following section.
\subsubsection{Player $i$'s problem in the DTDG with consistent conjectures}
Let $i\in \mathcal{N}$ be an arbitrary player in a DTDG with consistent conjectures.
Given $u^{-i} = (u^{-i}_1,\dots,u^{-i}_K) \in U^{-i}$, the actions of other players, $x^{-i} = (x^{-i}_2,\dots,x^{-i}_{K+1})\in X^{-i}$, the state conjectures of other players and $x_1$, the initial state of the game, player $i$'s problem in this game is denoted by $P^{'}_{i}(u^{-i},x^{-i};x_1$) which is given by the following.
$$
\problemsmall{$P^{'}_{i}(u^{-i},x^{-i};x_1$)}
{u^i,x^i}
{J^i(u^i,x^i;u^{-i},x_1)}
{(u^i,x^i) \in \Omega^{'}_i (u^{-i},x^{-i};x_1),}
$$
where,
\begin{align*}
{\Omega^{'}_{i}}({u^{-i},x^{-i};x_1})= \{(\hat{u}^{i},\hat{x}^{i})|\hat{u}^{i}_{k}\in U^{i}_{k}, \forall k \in \mathcal{K}, \hat{x}^i_{k+1}\in X^{i}_{k+1},\\ \forall k \in \mathcal{K}, \hat{x}^{i}_{2} = f_{1}(x_{1},\hat{u}^{i}_{1},u^{-i}_{1}), \hat{x}^{i}_{k+1} = f_{k}(\hat{x}^{i}_{k},\hat{u}^{i}_{k}, u^{-i}_{k}),\\\forall k \in \mathcal{K} \backslash \{1\}, \hat{x}^{i}_{k+1}=x^{j}_{k+1}, \forall j \in \mathcal{N}\backslash\{i\}, \forall k \in \mathcal{K}\}.
\end{align*}
The Cartesian product of the feasible set $ \Omega^{'}_{i} (u^{-i},x^{-i};x_1) $ over $i \in \mathcal{N}$ is given by the set value map denoted by $\Omega^{'}(u,x)$ which is given by,
$
{\Omega}^{'}(u,x)\triangleq \prod\limits_{i \in \mathcal{N}} {\Omega^{'}_{i}}(u^{-i},x^{-i};x_1),
$ where $u \in U$ and $x \in X$. The fixed points of the set value map $\Omega^{'}(u,x)$ is denoted as $\mathcal{F'}$ and is given by, $ \mathcal{F}' \triangleq \{(u,x)| (u,x)\in \Omega'(u,x) \}.$ It can be given by,
\begin{multline} \label{eqn:fixed_points_dash}
	\mathcal{F}^{'}=\{(u,x)| u\in U, x \in X, x^i_2=f_1(x_1,u^1_1,\dots,u^N_1), \forall i \in \mathcal{N},\\ x^i_{k+1}=f_k(x^i_k,u^1_k,\dots,u^N_k), \forall i \in \mathcal{N}, \forall k \in \mathcal{K}, \\
	x^i_{k+1}=x^j_{k+1}, \forall i \in \mathcal{N}, \forall j \in \mathcal{N}\backslash \{i\}, \forall k \in \mathcal{K} \}. 
\end{multline}
Notice that since the consistency condition is redundant in the set $\mathcal{F}^{'}$, the set $\mathcal{F}^{'}$ is same as $\mathcal{F}$. That is, $$\mathcal{F}^{'}=\mathcal{F}.$$ The definition of open-loop Nash equilibrium for a DTDG with consistent conjectures is given as follows.

\begin{definition} \label{def:consistent_conjecture_NE}
Consider a DTDG with consistent conjectures with the initial state $x_1$. A point $\{(u^{i*},x^{i*}); i \in \mathcal{N}\} \in \mathcal{F'}$ is said to be an open-loop Nash equilibrium of the DTDG with consistent conjectures if $\forall i \in \mathcal{N}$, given $(u^{-i*}, x^{-i*}) \in U^{-i} \times X^{-i}$, the actions and state conjectures at the equilibrium,
\begin{align*}
J^{i}(u^{i*},x^{i*};u^{-i*},x_1) \leq J^{i}&({u}^i,{x}^i;u^{-i*},x_1),\\&
\forall ({u}^i,{x}^i) \in \Omega^{'}_i (u^{-i*},x^{-i*};x_1).
\end{align*}
\end{definition}

\subsubsection{Relation to equilibria of the original game}
The relation of equilibria of the original game to equilibria of the game with consistent conjectures is given by the following proposition.
\begin{proposition}\label{relation_classical}
Consider a DTDG in state conjecture formulation. If $ (u^{1*},\dots,u^{N*},x^{1*},\dots,x^{N*}) \in \mathcal{F}$ is an open-loop Nash equilibrium of the original DTDG, then $(u^{1*},\dots,u^{N*},x^{1*},\dots,x^{N*}) \in \mathcal{F'}$ is also an open-loop Nash equilibrium of the DTDG with consistent conjectures.
\end{proposition}
\begin{proof}
Since $\mathcal{F}= \mathcal{F}^{'}$, an equilibrium of the original DTDG is also feasible for the DTDG with consistent conjectures. For player $i \in \mathcal{N}$, let $(u^{-i*},x^{-i*}) \in U^{-i} \times X^{-i}$ be the actions and state conjectures of other players at the equilibrium of the original game. Since the set $\Omega_i^{'}(u^{-i*},x^{-i*};x_1)$ has additional constraints than $\Omega_i(u^{-i*};x_1)$, for all $i \in \mathcal{N}$, $\Omega_i^{'}(u^{-i*},x^{-i*};x_1) \subseteq \Omega_i(u^{-i*};x_1)$. Hence  $(u^{1*},\dots,u^{N*},x^{1*},\dots,x^{N*}) \in \mathcal{F'}$ is also an open-loop Nash equilibrium of the DTDG with consistent conjectures.
\end{proof}
Thus, an equilibrium of the original game is also an equilibrium of the DTDG with consistent conjectures. The reverse relation will be discussed in Section \ref{subsec:epsilon_NE}. Next we show the shared constraint structure of the mapping $\Omega^{'}$ which is vital for the existence result.
\subsection{Shared constraint structure of the DTDG with consistent conjectures}
The motive of modifying the game with consistent conjectures is that it gives a shared constraint structure to the game which can be utilized for deriving an existence result. The relation of the consistent conjectures and the shared constraint structure applied to multi-leader multi-follower games is found in \cite{kulkarni2013consistency}. The set value map $\Omega^{'}(u,x)$ is a shared constraint mapping if there exists a set $\mathcal{S}$ such that $\forall i \in \mathcal{N}$,
\begin{equation*}
({u}^{i}, {x}^{i}) \in \Omega^{'}_{i}({u^{-i}, x^{-i}; x_1}) \Longleftrightarrow (u,x)\in \mathcal{S}.
\end{equation*}
Notice that the set $\mathcal{S}$ is independent of $i$.
The concept of shared constraint structure was studied in 1965 by Rosen \cite{rosen65existence} and is used in some recent papers (e.g., \cite{kulkarni09refinement},\cite{facchinei07ongeneralized},\cite{kulkarni12revisiting}). The following lemma shows that for a DTDG with consistent conjectures, the set value map $\Omega^{'}(u,x)$ has a shared constraint structure.

\begin{lemma}
\label{prop:shared_constraint}
Consider a DTDG with consistent conjectures. Let $\Omega^{'}(u,x)$ be the Cartesian product of the constraint set of problem ${P}^{'}_i (u^{-i},x^{-i};x_1$). Then $\Omega^{'}(u,x)$ is a shared constraint mapping with shared constraint set given by $\mathcal{F}^{'} \equiv ( U \times \mathcal{A}) \cap \mathcal{H}$, where
\begin{align*}
\mathcal{A} \triangleq \{({x}^{i},{x}^{-i})| ({x}^{i},{x}^{-i}) \in X, {x}^{i}_{k+1}=x^{j}_{k+1},\forall j \in \mathcal{N}\backslash\{i\}, \\\forall k \in \mathcal{K}\},
\\\text{and}
\mathcal{H} \triangleq \{({u}^{i},{u}^{-i},{x}^{i},{x}^{-i})| ({u}^{i},{u}^{-i})\in U, ({x}^{i},{x}^{-i})\in X, \\{x}^{i}_{2}= f_1(x_1,{u}^{i}_{1},{u}^{-i}_{1}),{x}^{i}_{k+1}= f_k({x}^{i}_k,{u}^{i}_{k},{u}^{-i}_{k}), \forall k \in \mathcal{K}\backslash \{1\}\}.
\end{align*}
\end{lemma}

\begin{proof}
Let $i \in \mathcal{N}$ be an arbitrary player for the game with consistent conjectures.
By using the defined sets $\mathcal{A}$ and $\mathcal{H}$, the set ${\Omega^{'}_{i}}({u^{-i},x^{-i};x_1})$ can be written as,
$
\Omega^{'}_{i}({u^{-i},x^{-i};x_1}) = \{(\hat{u}^{i},\hat{x}^{i}) | (\hat{u}^{i},u^{-i})\in U, (\hat{x}^{i},x^{-i})\in \mathcal{A}, (\hat{u}^{i},u^{-i},\hat{x}^{i},x^{-i}) \in \mathcal{H} \}.
$ i.e., $({u}^{i},{x}^{i}) \in \Omega^{'}_{i}({u^{-i},x^{-i};x_1}) \Longleftrightarrow (u,x)\in (U\times \mathcal{A})\cap \mathcal{H}$, which is independent of $i$. Since the implication holds for each $i \in \mathcal{N}$, $ \Omega^{'}(u,x)$ is a shared constraint mapping with shared constraint set given by $(U\times \mathcal{A})\cap \mathcal{H}$. It can be seen from \eqref{eqn:fixed_points_dash} that the shared constraint set is indeed the set of fixed points of the mapping $\Omega^{'}$. Hence, $\mathcal{F}^{'} = (U\times \mathcal{A})\cap \mathcal{H}$.
\end{proof}
Thus, we have,
\begin{equation} \label{eqn:shared_constraint}
({u}^{i},{x}^{i}) \in \Omega^{'}_{i}({u^{-i},x^{-i};x_1}) \Longleftrightarrow (u,x)\in \mathcal{F}^{'}.
\end{equation}
Now we use the relation \eqref{eqn:shared_constraint} to show the existence of equilibria for a class of games called the potential DTDGs in the next section.

\subsection{Existence of equilibria for a potential DTDG with consistent conjectures}
 Potential games were well studied by many and one pivotal paper is \cite{monderer96potential} by Monderer and Shapley.
We extend this concept to DTDGs and term it as potential DTDGs which is defined as follows.

\begin{definition}
Consider a DTDG in conjecture state formulation as defined in Section \ref{sec_equivalent_formulation} with cost function of players given by $\{J^1,J^2,\dots,J^N\}$. The DTDG is said to be potential if there exist a potential function $\pi$ such that for all $i \in \mathcal{N}$ and for all ${u}^{-i} \in U^{-i}$,
\begin{align*}
J^i(\tilde{u}^{i},\tilde{x}^{i},u^{-i};x_1) - J^i (\hat{u}^{i},\hat{x}^{i},u^{-i};x_1) =  \pi(\tilde{u}^{i},\tilde{x}^{i},u^{-i};x_1) - \\\pi (\hat{u}^{i},\hat{x}^{i},u^{-i};x_1),\forall (\tilde{u}^{i},\tilde{x}^{i}),(\hat{u}^{i},\hat{x}^{i}) \in U^i \times X^i.
\end{align*}
\end{definition}
The advantage of a potential game is that the minimizer of the potential function \cite{monderer96potential} over suitably defined feasible set gives an equilibrium of the game. Conditions for the existence of open-loop Nash equilibria in a potential DTDGs is given by considering the following optimisation problem denoted by $P$.
$$
\problemsmall{$P$}
{u,x}
{\pi(u,x;x_1)}
{(u,x) \in \mathcal{F}^{'}.}
$$
Note that the objective function of problem $P$ is the potential function, $\pi$ which is to be minimized over the shared constraint set $\mathcal{F}^{'}$. The following theorem shows that a minimizer of the problem $P$ is an open-loop Nash equilibrium of the potential DTDG with consistent conjectures.
\begin{theorem} \label{th:existence}
Consider a DTDG with consistent conjectures. Let $\mathcal{F}^{'}$ be the shared constraint set of the DTDG with consistent conjectures. Suppose the cost function of players admit a potential function $\pi$, then any minimizer of problem $P$ is an equilibrium of the potential DTDG with consistent conjectures.
\end{theorem}

\begin{proof}
Suppose $(u^{*},x^{*})\in \mathcal{F}^{'}$ is a global minimum of the problem $P$. Then,
$$
\pi(u^*,x^*;x_1) \leq \pi (u,x;x_1), \quad
\forall (u,x) \in \mathcal{F}^{'}.
$$
For some $i \in \mathcal{N}$, split the argument $u=(u^i,u^{-i}) \in U$ and $x=(x^i,x^{-i}) \in X$, we can rewrite the above inequality as,
$$
\pi(u^*,x^*;x_1) \leq \pi (u^i,u^{-i},x^{i},x^{-i};x_1), \forall (u^i,u^{-i},x^{i},x^{-i}) \in \mathcal{F}^{'}.
$$
The inequality still holds even if we replace $u^{-i}$ and $x^{-i}$ by $u^{-i*}$ and $x^{-i*}$, the actions and state conjectures of other players at the minimum. That is,
\begin{align*}
\pi(u^*,x^*;x_1) \leq \pi (u^{i},u^{-i*},x^{i},&x^{-i*};x_1),
 \\ &\forall (u^{i},u^{-i*},x^{i},x^{-i*}) \in \mathcal{F}^{'}.
\end{align*}
Since $\mathcal{F}^{'}$ is a shared constraint mapping of $\Omega^{'}(u,x)$, by Lemma \ref{prop:shared_constraint},
$(u^{i},u^{-i*},x^{i},x^{-i*}) \in \mathcal{F}' \Longrightarrow (u^{i},x^{i}) \in \Omega_i^{'}(u^{-i*},x^{-i*};x_1). $
Hence, we can rewrite the inequality as,
\begin{align*}
\pi(u^*,x^*;x_1) \leq \pi (u^{i},u^{-i*},x^{i},&x^{-i*};x_1),
\\ &\forall (u^{i},x^{i}) \in \Omega_i^{'}(u^{-i*},x^{-i*};x_1).
\end{align*}
Since $\pi$ is a potential function of the DTDG,
$\forall i \in \mathcal{N}$,
\begin{align*}
J^i(u^{i*},x^{i*},u^{-i*},x^{-i*};x_1) \leq J^i (u^{i},x^{i},u^{-i*},x^{-i*};x_1), \\ \forall (u^{i},x^{i}) \in \Omega_i^{'}(u^{-i*},x^{-i*};x_1).
\end{align*}
Hence, $(u^*,x^*) \in \mathcal{F'}$ is an open-loop Nash equilibrium of the potential DTDG with consistent conjectures.
\end{proof}

Thus, for a potential DTDG with consistent conjectures, a minimiser of the problem $P$ is an open-loop Nash equilibrium. Conditions under which an equilibrium exist for a potential DTDG with consistent conjectures is given in the following corollaries.
\begin{corollary}
\label{lm:compact}
Consider a potential DTDG with consistent conjectures. Suppose the potential function $\pi$ is continuous and the set $\mathcal{F}^{'}$ is compact, then the game has an equilibrium given by the minimizer of the problem $P$. The shared constraint set, $\mathcal{F}^{'}$ is compact if the sets $U,X$ are compact and $f_k$, the state mapping is continuous for all $k \in \mathcal{K}$.
\end{corollary}

Conditions for the existence of an equilibrium without the compactness of state space and action set is given as follows.
\begin{corollary}
\label{lm:compact}
Consider a potential DTDG with consistent conjectures. Suppose the potential function $\pi$ is continuous and coercive and $f_k$ is continuous for all $k \in \mathcal{K}$, then the game has an equilibrium.
\end{corollary}
Thus, for a potential DTDG with consistent conjectures under these conditions we can guarantee the existence of an open-loop Nash equilibrium and the equilibrium is given by a minimizer of problem $P$.

\subsection{$\epsilon$-Nash equilibrium of the original game} \label{subsec:epsilon_NE}
We have shown that an equilibrium of the original game is an equilibrium of the game with consistent conjectures. In this section, the reverse relation of the equilibrium is being considered, i.e., an equilibrium of the DTDG with consistent conjectures is related to $\epsilon$-Nash equilibrium of the original DTDG. The definition of an $\epsilon$-Nash equilibrium of the original game is given as follows.    
\begin{definition}
Consider a DTDG in state conjecture formulation as defined in Section \ref{sec_equivalent_formulation}. For $\epsilon \geq 0$, a point $\{u^{i*},x^{i*}; i \in \mathcal{N}\} \in (U \times X)$ is an $\epsilon$-Nash equilibrium of the original game if $\forall i \in \mathcal{N}$, given $u^{-i*} \in U^{-i}$ and $x_1$,
\begin{multline*}
J^{i}(u^{i*},x^{i*};u^{-i*},x_1) \leq  \underset{\tilde{u}^i,\tilde{x}^i}{\inf}  J^{i}(\tilde{u}^i,\tilde{x}^i;u^{-i*},x_1) + \epsilon,
\end{multline*}
where the $\inf$ is over ($\tilde{u}^i,\tilde{x}^i)\in \Omega_i (u^{-i*};x_1)$.
\end{definition} 
  In this section, we first analyse the case of linear-quadratic DTDGs with a convex cost functions and then we extend the analysis for a general class of DTDGs. 

    \subsubsection{$\epsilon$-Nash equilibrium for linear-quadratic DTDGs }
        
    The aim of this section is to relate a particular equilibrium of the linear-quadratic DTDG with consistent conjectures to an $\epsilon$-Nash equilibrium of the original game.    
    We use a result on the \textit{exact penalty functions} from the book \cite{nocedal99numerical} to find the $\epsilon$-Nash equilibrium relation for linear-quadratic DTDGs. In order to state the result, consider a general optimisation problem denoted by $\mathbb{P}$ with equality constraints.
        $$
        \problemsmall{$\mathbb{P}$}
        {x\in \mathbb{R}^n}
        {f(x)}
        {h_i(x)=0, i= 1,\dots, p,}
        $$
        where $f:\mathbb{R}^n\rightarrow \mathbb{R}$ and $h_i:\mathbb{R}^n\rightarrow \mathbb{R}^p$, $p\leq n$ are continuously differentiable functions. The following proposition provides the result that we use to derive the $\epsilon$-Nash equilibrium relation.
        
        \begin{proposition}(\textit{Theorem 17.3} in \cite{nocedal99numerical})\label{nocedal_result} \label{nocedal_result}
            Suppose $x^*$ is a strict local solution of the nonlinear programming problem $\mathbb{P}$ at which the first order KKT conditions are satisfied with Lagrange multipliers $\lambda_i^{*}, i=1,\dots,p$. Then $x^*$ is a local minimizer of $\phi (x;\mu)$ for all $\mu > \mu^*$, where $\mu^*= \underset{i=1,\dots,p}{\max}|\lambda_i^{*}|$ and $\phi(x;\mu)= f(x)+ \mu \sum\limits_{i=1,\dots,p}^{}|h_i(x)|$.
        \end{proposition}
        
    It can be easily verified that if the problem $\mathbb{P}$ is a convex optimization problem, then the result in Proposition \ref{nocedal_result} is also valid for a global solution. Consider a linear-quadratic DTDG as defined in Definition \ref{def_LQ} with $Q^i_{k+1} \succ 0, \forall i \in \mathcal{N}, \forall k \in \mathcal{K}$. Therefore, a player's problem in a linear-quadratic DTDG is a convex optimization problem. Let $(u^*,x^*)\in U\times X$ be an equilibrium of the linear-quadratic DTDG with consistent conjectures that we are interested in. For player $i \in \mathcal{N}$, let $(u^{-i*},x^{-i*})\in U^{-i}\times X^{-i}$ be the actions and state conjectures of other players at the equilibrium of the linear-quadratic DTDG with consistent conjectures. Given $(u^{-i*},x^{-i*})$, the player $i$'s problem in a linear-quadratic DTDG with consistent conjectures is given by the following  problem which is denoted by $L_i(u^{-i*},x^{-i*};x_1)$.
        \begin{align*}
        {\underset{{u}^i,x^i}{\min}} \sum\limits_{k\in\mathcal{K}}&\frac{1}{2}\Bigl( x^{i'}_{k+1}Q^i_{k+1}x^i_{k+1}+ u^{i'}_k R^{ii}_k u^{i}_{k} +  \sum\limits_{j \in \mathcal{N}\backslash\{i\}}u^{j^{*'}}_kR^{ij}_ku^{j*}_k\Bigr) \\ &
        \text{s.t.}     {({u}^i,x^i) \in \Omega_i^{'}(u^{-i*},x^{-i*};x_1), }
        \end{align*}
            where, in this case,
            \begin{align*}
                \Omega_i^{'}(u^{-i*},x^{-i*};x_1) = \{u^i_1,\dots,u^i_K,x^i_2,\dots,x^i_{K+1}| u^i_k \in U^i_k,\\ \forall k \in \mathcal{K}, x^i_{k+1}\in X^{i}_{k+1}, \forall k \in \mathcal{K}, x^i_{k+1}= A_kx^i_k+b^i_k u^i _k +  \\ \sum\limits_{j\in\mathcal{N}\backslash \{i\}}b^j_ku^{j*}_k, \forall k \in \mathcal{K}, x^i_{k+1}=x^{j*}_{k+1}, \forall j \in \mathcal{N}\backslash\{i\}, \forall k \in \mathcal{K} \}. 
                \end{align*}
            
            Here also $u^i =(u^i_1,\dots, u^i_K)$ and $x^i=(x^i_2,\dots,x^i_{K+1})$. Consider the player $i$'s problem with an exact penalty function for the consistent state conjecture constraint with penalty parameter $\mu^{i}$. Let us denote this problem by $\tilde{L}_i(u^{-i*},x^{-i*};\mu^i)$ and is given by the following problem. 
            \begin{align*}
            \underset{{u}^i,{x}^i}{\min} \sum\limits_{k\in\mathcal{K}}\frac{1}{2}\Bigl( x^{i'}_{k+1}Q^i_{k+1}x^i_{k+1}+ u^{i'}_k R^{ii}_k u^{i}_{k} +& \\  \sum\limits_{j \in \mathcal{N}\backslash\{i\}}u^{j^{*'}}_kR^{ij}_ku^{j*}_k + \mu^{i} \sum\limits_{j \in \mathcal{N}\backslash\{i\}} &|x^i_{k+1} - x^{j*}_{k+1}| \Bigr) \\ \text{s.t.} {({u}^i,{x}^i) \in \Omega_i(u^{-i*};x_1)}, \text{where,} 
    \end{align*}
        \begin{align*} \Omega_i(u^{-i*};x_1)= \{u^i_1,\dots,u^i_K,x^i_2,\dots,x^i_{K+1}| u^i_k \in U^i_k, \\ \forall k \in \mathcal{K}, x^i_{k+1}\in X^{i}_{k+1}, \forall k \in \mathcal{K}, \\x^i_{k+1}= A_kx^i_k+b^i_k u^i _k +  \sum\limits_{j\in\mathcal{N}\backslash \{i\}}b^j_ku^{j*}_k, \forall k \in \mathcal{K} \}.
        \end{align*}
    
        Thus, we have two problems to represent the player $i$'s problem in a linear-quadratic DTDG with consistent conjectures. One is $L_i(u^{-i*},x^{-i*};x_1)$ with the consistent conjecture condition in the constraint set and another is  $\tilde{L}_i(u^{-i*},x^{-i*};\mu^i)$ with exact penalty function for the consistent conjecture condition. Though the problems are different, the equivalence of equilibria of games having these problems can be derived directly from Proposition \ref{nocedal_result} and is stated as follows. 
        \begin{proposition}\label{prop:epsilon_LQ}
            Consider a linear-quadratic DTDG with consistent conjectures with $Q^i_{k+1} \succ 0, \forall i \in \mathcal{N}, \forall k \in \mathcal{K}$. Suppose $(u^{*},x^{*}) \in \mathcal{F}^{'}$ is an open-loop Nash equilibrium of the game. For $i \in \mathcal{N}$, let $(u^{-i*},x^{-i*}) \in U^{-i}\times X^{-i}$ be the actions and state conjectures of other players at this equilibrium. Let the player $i$'s problem in a linear-quadratic DTDG with consistent conjecture be given by $L_i (u^{-i*},x^{-i*};x_1)$.
             Suppose the Lagrange multipliers at the equilibrium of the game corresponding to the consistent state conjecture constraints are given by $\lambda^{ij*}_{k+1}, \forall j \in \mathcal{N} \backslash \{i\}, \forall k \in \mathcal{K}$. Then $(u^{*},x^{*})$ is also an equilibrium of the game with the player $i$'s problem given by $\tilde{L}_i(u^{-i*},x^{-i*};\mu^i)$ for all $\mu^{i} > \mu^{i*}$, where, 
            \begin{equation}\label{penalty_value}
            \mu^{i*} = \underset{k \in \mathcal{K}}{\max} \underset{j\in\mathcal{N}\backslash \{i\} }{\max}|\lambda^{ij*}_{k+1}|.
            \end{equation}
            \end{proposition}
            \begin{proof}
                Since $Q^i_{k+1} \succ 0, \forall i \in\mathcal{N}, \forall k \in \mathcal{K}$, the problem $L_i(u^{-i*},x^{-i*};x_1)$ is a convex optimisation problem which has a global solution. By Proposition \ref{nocedal_result}, the result follows.
                \end{proof}
             We now use Proposition \ref{prop:epsilon_LQ} to derive the result which relates the equilibrium of a linear-quadratic DTDG with consistent conjectures to $\epsilon$-Nash equilibrium of the original linear-quadratic DTDG.  Let $(u^*,x^*) \in \mathcal{F}^{'}$ be an equilibrium of the linear-quadratic DTDG with consistent conjectures.
            Suppose player $i \in \mathcal{N}$ plays the best response in the original game to $(u^{-i*},x^{-i*})$. Then in that case let us denote $\mathbf{{x}}^i_{k+1}, k \in \mathcal{K}$ as the best response state trajectory conjectured by player $i$ in the original game. The relation of $(u^{*},x^{*})$ to the $\epsilon$-Nash equilibrium of the original game is stated as follows.
            \begin{theorem}\label{LQ_epsilon}
            Consider a linear-quadratic DTDG with consistent conjectures with $Q^i_{k+1} \succ 0, \forall i \in \mathcal{N}, \forall k \in \mathcal{K}$. Suppose $(u^{*},x^{*}) \in \mathcal{F}^{'}$ is an equilibrium of this game. For player $i \in \mathcal{N}$, let $x^{j*}_{k+1}\in X^j_{k+1}, k \in \mathcal{K}$ be the state conjecture of player $j\in \mathcal{N}\backslash\{i\}$ at the equilibrium of this game. Let $\lambda^{ij*}_{k+1}, \forall j \in \mathcal{N} \backslash \{i\}, \forall k \in \mathcal{K}$ be as in Proposition \ref{prop:epsilon_LQ}. Let $\mathbf{{x}}^i_{k+1}, k \in \mathcal{K}$ be the player $i$'s best response state conjecture in the original game corresponding to $u^{-i*} \in U^{-i}$.
           Then $(u^{*},x^{*})$ is an $\epsilon$-Nash equilibrium of the original game with $\epsilon= \underset{i \in \mathcal{N}}{\max} \{\epsilon_i\}$, where $\epsilon_i$ satisfy $\sum\limits_{k \in \mathcal{K}} \sum\limits_{j \in \mathcal{N}\backslash\{i\}} \mu^{i} |\mathbf{{x}}^i_{k+1} - x^{j*}_{k+1}| \leq 2 \epsilon_i$ and $\mu^{i} > \mu^{i*}$, where $\mu^{i*}$ is given by \eqref{penalty_value}.
            \end{theorem}
            
            \begin{proof}
                Suppose $\{(x^{i*},u^{i*}), i \in \mathcal{N}\}$ is an open-loop Nash equilibrium of the linear-quadratic DTDG with consistent conjectures. That is for all $i \in \mathcal{N}$ and for each $\mu^{i} > \mu^{i*}$,
                \begin{multline*}
                 \sum\limits_{k\in\mathcal{K}}\frac{1}{2}\Bigl( x^{i{*'}}_{k+1}Q^i_{k+1}x^{i*'}_{k+1}+ \sum\limits_{j \in \mathcal{N}}u^{j*'}_kR^{ij}_ku^{j*}_k  \Bigr)    \leq \\
                  \sum\limits_{k\in\mathcal{K}}\frac{1}{2}\Bigl( x^{i'}_{k+1}Q^i_{k+1}x^i_{k+1}+  u^{i'}_kR^{ii}_ku^i_k+\sum\limits_{j \in \mathcal{N}\backslash \{i\}}u^{j*'}_kR^{ij}_ku^{j*}_k +\\ \sum\limits_{j \in \mathcal{N}\backslash\{i\}}\mu^{i} |x^i_{k+1} - x^{j*}_{k+1}| \Bigr),
                   \forall     ({u}^i,x^i) \in     \Omega_i(u^{-i*};x_1).            \end{multline*}
                Let $(\mathbf{{u}}^i_1,\dots,\mathbf{{u}}^i_K,\mathbf{{x}}^i_{2},\dots,\mathbf{{x}}^i_{K+1}) \in     \Omega_i(u^{-i*};x_1)$ be the best response of the original linear-quadratic DTDG corresponding to $u^{-i*}$, the actions of other players at the equilibrium of the game with consistent conjectures. Hence,
                \begin{multline*}
                    \sum\limits_{k\in\mathcal{K}}\frac{1}{2}\Bigl( x^{i{*'}}_{k+1}Q^i_{k+1}x^{i*'}_{k+1}+ \sum\limits_{j \in \mathcal{N}}u^{j*'}_kR^{ij}_ku^{j*}_k  \Bigr)    \leq \\
                    \sum\limits_{k\in\mathcal{K}}\frac{1}{2}\Bigl( \mathbf{{x}}^{i'}_{k+1}Q^i_{k+1}\mathbf{{x}}^i_{k+1}+ \mathbf{{u}}^{i'}_kR^{ii}_k\mathbf{{u}}^i_k +\sum\limits_{j \in \mathcal{N}\backslash \{i\}}u^{j*'}_kR^{ij}_ku^{j*}_k + \\ \sum\limits_{j \in \mathcal{N}\backslash\{i\}}\mu^{i} |\mathbf{{x}}^i_{k+1} - x^{j*}_{k+1}| \Bigr).            \end{multline*}
                By rearranging, the above inequality can be written as,
               \begin{multline*}
                   \sum\limits_{k\in\mathcal{K}}\frac{1}{2}\Bigl( x^{i{*'}}_{k+1}Q^i_{k+1}x^{i*'}_{k+1}+ \sum\limits_{j \in \mathcal{N}}u^{j*'}_kR^{ij}_ku^{j*}_k  \Bigr)    \leq \\
                   \sum\limits_{k\in\mathcal{K}}\frac{1}{2}\Bigl( \mathbf{{x}}^{i'}_{k+1}Q^i_{k+1}\mathbf{{x}}^i_{k+1}+ \mathbf{{u}}^{i'}_kR^{ii}_k\mathbf{{u}}^i_k +\sum\limits_{j \in \mathcal{N}\backslash \{i\}}u^{j*'}_kR^{ij}_ku^{j*}_k \Bigr) + \\ \frac{1}{2} \sum\limits_{k\in\mathcal{K}}\sum\limits_{j \in \mathcal{N}\backslash\{i\}}\mu^{i} |\mathbf{{x}}^i_{k+1} - x^{j*}_{k+1}|.            \end{multline*}
                Since $\sum\limits_{k\in\mathcal{K}}\sum\limits_{j \in \mathcal{N}\backslash\{i\}}\mu^{i} |\mathbf{{x}}^i_{k+1} - x^{j*}_{k+1}| \leq 2 \epsilon_i$, $\forall i \in \mathcal{N}$ and $(\mathbf{{u}}^i_1,\dots,\mathbf{{u}}^i_K,\mathbf{{x}}^i_{2},\dots,\mathbf{{x}}^i_{K+1})$  is the best response of the original game,
   \begin{multline*}
   \sum\limits_{k\in\mathcal{K}}\frac{1}{2}\Bigl( x^{i{*'}}_{k+1}Q^i_{k+1}x^{i*'}_{k+1}+ \sum\limits_{j \in \mathcal{N}}u^{j*'}_kR^{ij}_ku^{j*}_k  \Bigr)    \leq \\
                        \underset{({u}^i,{x}^i)\in \Omega_i(u^{-i*};x_1)    }{\inf} \sum\limits_{k\in\mathcal{K}}\frac{1}{2}\Bigl( x^{i'}_{k+1}Q^i_{k+1}x^i_{k+1}+ {u}^{i'}_kR^{ii}_k{u}^i_k  \\+\sum\limits_{j \in \mathcal{N}\backslash \{i\}}u^{j*'}_kR^{ij}_ku^{j*}_k \Bigr)+ \epsilon_i \end{multline*}
                Hence, $\{(x^{i*},u^{i*}); i \in \mathcal{N}\}$ is an $\epsilon$- Nash equilibrium of linear-quadratic DTDG with $\epsilon= \max \{\epsilon_i, i \in \mathcal{N}\}$.
                \end{proof}
                Thus, under conditions specified in Theorem \ref{LQ_epsilon}, an equilibrium of a linear-quadratic DTDG with consistent conjectures is an $\epsilon$-Nash equilibrium of the original linear-quadratic DTDG for a certain $\epsilon$.    Note that the convexity of the cost function and constraint set is only used in Theorem \ref{LQ_epsilon}. Hence the result can be extended to any convex problems, not limited to linear-quadratic games. However, for a general case convexity need not hold. Therefore the next section provides an $\epsilon$-Nash equilibrium relation for a general case, but with additional assumptions.

\subsubsection{$\epsilon$-Nash equilibrium for general DTDGs}
 For a general class of games, we use a result on the \textit{exact penalty functions} from the paper \cite{di1989exact} by Di Pillo and Grippo for deriving the relation of an equilibrium of the game with  consistent conjectures to the $\epsilon$-Nash equilibrium of the original game. The Mangasarian-Fromowitz constraint qualification (MFCQ) (See Definition \ref{def:MFCQ}) needs to be satisfied for the result to hold. For that, we consider an equivalent game having a minor modification from the DTDG with consistent conjectures and we will refer to it as \textit{alternate consistent conjecture DTDG}. In this alternate consistent conjecture DTDG we define an \textit{exact penalty function game} having an exact penalty function for the consistent conjecture condition. We derive the relation of these two games, the alternate consistent conjecture DTDG and the exact penalty function game using the result from \cite{di1989exact}. The $\epsilon$-Nash equilibrium condition is provided utilising the relation of the exact penalty function game in turn to the DTDG with consistent conjectures. 

In order to make the explanation clear, let us denote the original game, the DTDG with consistent conjectures and the alternate consistent conjecture DTDG by $\mathcal{G}_o$, $\mathcal{G}_c$ and $\mathcal{G}_a$  respectively. The difference of $\mathcal{G}_a$, the alternate consistent conjecture DTDG  and $\mathcal{G}_c$, the DTDG with consistent conjecture is that a player's state conjecture is constrained to be consistent with the state conjecture of one other player in the game (in this case, the next labelled player), not all players as in $\mathcal{G}_c$.

 Let $i \in \mathcal{N}$ be a player in a DTDG. Given $u^{-i}\in U^{-i},x^{-i}\in X^{-i}$ and $x_1$, the player $i$'s problem in $\mathcal{G}_a$ is denoted by $\bar{{P}}_i$($u^{-i},x^{-i};x_1$) and is given by the following.
$$
\problemsmall{$\bar{{P}}_i$($u^{-i},x^{-i};x_1$)}
{u^i,x^i}
{J^i(u^i,x^i;u^{-i},x_1)}
{(u^i,x^i) \in \bar{\Omega}_i (u^{-i},x^{-i};x_1),}
$$
where,
\begin{align*}
{\bar{\Omega}_{i}}({u^{-i},x^{-i};x_1})= \{\hat{u}^{i}_1,\dots,\hat{u}^{i}_{K},\hat{x}^{i}_{2},\dots,\hat{x}^{i}_{K+1}|\hat{u}^{i}_{k}\in U^{i}_{k},\\\forall k \in \mathcal{K}, \hat{x}^{i}_{k+1}\in X^{i}_{k+1}, \forall k \in \mathcal{K}, \hat{x}^{i}_{2} = f_{1}(x_{1},\hat{u}^{i}_{1},u^{-i}_{1}), \\ \hat{x}^{i}_{k+1} = f_{k}(\hat{x}^{i}_{k},\hat{u}^{i}_{k}, u^{-i}_{k}), \forall k \in \mathcal{K}  \backslash \{1\}, \hat{x}^{i}_{k+1}=x^{j(i)}_{k+1}, \forall k \in \mathcal{K}\},
\end{align*}
 where the \textit{next labelled player} $j(i) \in \mathcal{N}$ is given by
\begin{equation} \label{next_labelled_player}
{j(i)} =
\left\{
\begin{aligned}
& i+1,  &&  i \in \mathcal{N}\backslash\{N\}
\\[1ex]
& 1,  && i = N.
\end{aligned}
\right.    
\end{equation}

 Using the player $i$'s problem given by $\bar{{P}}_i$($u^{-i},x^{-i};x_1$), the open-loop Nash equilibrium in $\mathcal{G}_a$ can be defined as follows.

\begin{definition}  \label{def:alternate_cons_NE}
    An $N$-tuple of strategies \{$(u^{i*},x^{i*}) \in U^i\times X^i;  i \in \mathcal{N}$\} is said to be an open-loop Nash equilibrium of $\mathcal{G}_a$ if $\forall i \in \mathcal{N}$, given $(u^{-i*},x^{-i*}) \in U^{-i}\times X^{-i}$, the actions and state conjectures of other players at equilibrium of $\mathcal{G}_a$, 
    \begin{align*}
    J^{i}({u}^{i*},{x}^{i*};u^{-i*},x_1) \leq  J^{i}&(\tilde{u}^i,\tilde{x}^i;u^{-i*},x_1),\\&
    \forall (\tilde{u}^i,\tilde{x}^i) \in \bar{\Omega}_i (u^{-i*},x^{-i*};x_1).
\end{align*}
\end{definition}
The equivalence of $\mathcal{G}_c$ and $\mathcal{G}_a$ in terms of the equilibrium set is given in the following proposition.

\begin{proposition}\label{equivalence_equi}
        Consider a DTDG in state conjecture formulation. Suppose \{$(u^{i*},x^{i*}) \in U^i\times X^i; i \in \mathcal{N}$\} is an open-loop Nash equilibrium of $\mathcal{G}_c$, the DTDG with consistent conjectures, then \{$(u^{i*},x^{i*}) \in U^i\times X^i; i \in \mathcal{N}$\} is also an equilibrium of $\mathcal{G}_a$, alternate consistent conjecture DTDG and vice versa.
    \end{proposition}
\begin{proof}
    $``\Rightarrow"$ Let \{$(u^{i*},x^{i*}) \in U^i\times X^i; i \in \mathcal{N}$\} be an open-loop Nash equilibrium of $\mathcal{G}_c$. Then from Definition \ref{def:consistent_conjecture_NE},
at an equilibrium $\{(u^{i*},x^{i*}), i \in \mathcal{N}\}$ of $\mathcal{G}_c$, since all players conjectures consistently, the consistency condition with $N-1$ players in the set $\Omega^{'}_{i}({u^{-i*},x^{-i*};x_1}) $ can be replaced with a single player condition. Hence, at the equilibrium $\{(u^{i*},x^{i*}), i \in \mathcal{N}\}$ of $\mathcal{G}_c$, $\Omega^{'}_{i}({u^{-i*},x^{-i*};x_1}) = \bar{\Omega}_{i}({u^{-i*},x^{-i*};x_1)}$. Thus, $\{(u^{i*},x^{i*}), i \in \mathcal{N}\}$ is also an equilibrium of $\mathcal{G}_a$.

     $``\Leftarrow"$ Suppose \{$(u^{i\diamond},x^{i\diamond}) \in U^i\times X^i; i \in \mathcal{N}$\} is an equilibrium of $\mathcal{G}_a$. Then from Definition \ref{def:alternate_cons_NE},
         at an equilibrium \{$(u^{i\diamond},x^{i\diamond}); i \in \mathcal{N}$\} of $\mathcal{G}_a$, for all $i \in \mathcal{N}$, $x^{i\diamond}_{k+1}=x^{j\diamond}_{k+1}, \forall j \in \mathcal{N}\backslash\{i\},\forall k \in \mathcal{K}.$ Hence, at the equilibrium $\{(u^{i\diamond},x^{i\diamond}), i \in \mathcal{N}\}$ of $\mathcal{G}_a$, the consistent condition with the adjacent player in $\bar{\Omega}_{i}({u^{-i\diamond},x^{-i\diamond};x_1})$ can be replaced with the consistent condition for all rivals. i.e. at the equilibrium \{$(u^{i\diamond},x^{i\diamond}); i \in \mathcal{N}$\},  $\bar{\Omega}_{i}({u^{-i\diamond},x^{-i\diamond};x_1}) = {\Omega}'_{i}({u^{-i\diamond},x^{-i\diamond};x_1})$ and hence $\{(u^{i\diamond},x^{i\diamond}), i \in \mathcal{N}\}$ is also an equilibrium of $\mathcal{G}_c$.
    \end{proof}
    
Thus, by Proposition \ref{equivalence_equi}, the set of equilibria of both games, $\mathcal{G}_c$ and $\mathcal{G}_a$ are equivalent. Now consider a particular equilibrium $\{(u^{i*},x^{i*}), i \in \mathcal{N}\}$ of $\mathcal{G}_c$. Our aim is to relate this equilibrium to an $\epsilon$-Nash equilibrium of $\mathcal{G}_o$, the original game.
Let $(u^{-i*},x^{-i*}) \in U^{-i}\times X^{-i}$ be the actions and state conjectures of the rivals of player $i \in \mathcal{N}$ at the equilibrium of $\mathcal{G}_c$. Given $u^{-i*}$, the best response set of player $i$ in  $\mathcal{G}_o$ is denoted by ${\mathcal{R}}_i(u^{-i*})$ and is given by the following.
    \begin{align*}
    {\mathcal{R}}_i&(u^{-i*})=\{(u^i,x^i)|\\&(u^i,x^i)\in \arg \underset{(\hat{u}^i,\hat{x}^i)\in {\Omega}_{i}({u^{-i*};x_1})}{\min} J^i(\hat{u}^i,\hat{x}^i;u^{-i*},x_1) \}.
    \end{align*}
    
 Let ${\bar{\mathcal{R}}}_i(u^{-i*},x^{-i*})$ denote the best response of player $i$ in $\mathcal{G}_a$ corresponding to $(u^{-i*},x^{-i*})$ which is given by the following.
    \begin{align*}
        &\bar{{\mathcal{R}}}_i(u^{-i*},x^{-i*})=\{(u^i,x^i)|\\&(u^i,x^i)\in \arg \underset{(\hat{u}^i,\hat{x}^i)\in {\bar{\Omega}}_{i}({u^{-i*},x^{-i*};x_1})}{\min} J^i(\hat{u}^i,\hat{x}^i;u^{-i*},x_1) \}.
        \end{align*}
    
        In order to make the structure of the problem same as the one used by Di Pillo and Grippo in \cite{di1989exact}, we assume ${\mathcal{R}}_i (u^{-i*})$ is bounded and we introduce a compact set $\mathcal{D}_i = \mathcal{D}_i(u^{-i*}), i \in \mathcal{N}$ (with non-empty interior) such that the best response set ${\mathcal{R}}_i (u^{-i*}) \subset \mathring{\mathcal{D}}_i$. The following proposition shows that the best response set of alternate consistent conjecture DTDG $\bar{{\mathcal{R}}}_i(u^{-i*},x^{-i*})$ is also contained in the compact set $\mathring{\mathcal{D}}_i$.
    
    \begin{proposition}
        Consider an equilibrium \{$(u^{i*},x^{i*}) \in U^i\times X^i; i \in \mathcal{N}$\} of $\mathcal{G}_c$, the DTDG with consistent conjectures. Let $(u^{-i*},x^{-i*}) \in U^{-i} \times X^{-i}$ be the actions and state conjectures of rivals of player $i$ at the equilibrium of $\mathcal{G}_c$. Let ${\mathcal{R}}_i(u^{-i*})$ and ${\bar{\mathcal{R}}}_i(u^{-i*},x^{-i*})$ denote the set of best responses of player $i$ in the original game $\mathcal{G}_o$ and the alternate consistent conjecture game $\mathcal{G}_a$ respectively corresponding to $(u^{-i*},x^{-i*})$. If the best response set ${\mathcal{R}}_i(u^{-i*})$ is contained in the compact set $\mathcal{D}_i$, then the best response set ${\bar{\mathcal{R}}}_i(u^{-i*},x^{-i*})$ is also contained in $\mathcal{D}_i$.
    \end{proposition}

   \begin{proof}
   	 Since there is one additional constraint of consistency of conjectured state in a player's problem in $\mathcal{G}_a$ than in $\mathcal{G}_o$, $\bar{\Omega}^{'}_i \subseteq \Omega_i$ and hence $\bar{\mathcal{R}}_i(u^{-i*},x^{-i*}) \subseteq {\mathcal{R}}_i(u^{-i*})$. Therefore, ${\bar{\mathcal{R}}}_i(u^{-i*},x^{-i*})$ is contained in $\mathring{\mathcal{D}}_i$.
    \end{proof}
    
     Let ($u^*,x^*$) be an equilibrium of $\mathcal{G}_a$. Given $(u^{-i*},x^{-i*}) \in U^{-i} \times X^{-i}$, consider a game with the following problem for  player $i$ denoted by $\tilde{{P}}_i$($u^{-i*},x^{-i*};x_1$).
    $$
    \problemsmall{$\tilde{{P}}_i$($u^{-i*},x^{-i*};x_1$)}
    {u^i,x^i}
    {J^i(u^i,x^i;u^{-i*},x_1)}
    {(u^i,x^i) \in \bar{\Omega}_i (u^{-i*},x^{-i*};x_1)\cap {\mathring{\mathcal{D}}}_i.}
    $$
Note that the feasible set in this game is an intersection of the interior of the compact set $\mathcal{D}_i$ and the feasible set of $\bar{{P}}_i$($u^{-i*},x^{-i*};x_1)$. Since $\mathring{\mathcal{D}}_i$ contains the best response of player $i$ in $\mathcal{G}_a$, the solution of problem $\bar{{P}}_i$($u^{-i*},x^{-i*};x_1)$ is also a solution of $\tilde{{P}}_i$($u^{-i*},x^{-i*};x_1)$. Therefore, an equilibrium of $\mathcal{G}_a$ is also an equilibrium of the game with player $i$'s problem given by $\tilde{{P}}_i(u^{-i*},x^{-i*};x_1)$.
Now consider the player $i$'s problem given by $\tilde{{P}}_i(u^{-i*},x^{-i*};x_1$) with an \textit{exact penalty function} for the consistent state conjecture constraint with penalty parameter $\mu_i$. Let us denote this problem by ${{Q}}_i (u^{-i*},x^{-i*};\mu_i$) which is given by the following.
$$
\problemsmall{}
{u^i,x^i}
{J^i(u^i,x^i;u^{-i*},x_1)+{\mu_i}|x^i-x^{j(i)*}|}
{(u^i,x^i) \in \Omega_i (u^{-i*};x_1) \cap {\mathring{\mathcal{D}_i},}}
$$
 where $\mu_i > 0$ is the penalty parameter and $x^{j(i)*} \in X^{j(i)}$ is the state conjecture of the next labelled player at the equilibrium of $\mathcal{G}_c$, where $j(i)$ given by \eqref{next_labelled_player}. We will refer to the game with player $i$'s problem given by ${{Q}}_i$($u^{-i*},x^{-i*};\mu_i$) as \textit{exact penalty function game}. The relation of equilibria of $\mathcal{G}_a$ and exact penalty function game can be deduced directly from a result in the paper \cite{di1989exact} which is given in Proposition \ref{prop_MFCQ}.
 For the result to hold, the Mangasarian-Fromowitz constraint qualification (MFCQ) is required to satisfy. For defining MFCQ, consider the equality constrained optimisation problem given by $\mathbb{P}$.
\begin{definition} {\label{def:MFCQ}}
    The MFCQ holds at $x\in \mathbb{R}^n$ for problem $\mathbb{P}$ if $\nabla h_j(x), j=1,\dots, p$ are linearly independent and there exist a $z \in \mathbb{R}^n$ such that  $\nabla h_j(x)'z=0, j=1,\dots, p.$
    \\By the theorem of alternatives, the MFCQ can be restated as follows.
    The MFCQ holds at $x\in \mathbb{R}^n$ if there exist no $v_j,j=1, \dots, p$ such that $
    \sum\limits_{j=1}^{p} v_j \nabla h_j(x) = 0, (v_j, j= 1, \dots, p) \neq 0.
    $
    \end{definition}
    Note that MFCQ is satisfied at the equilibrium of $\mathcal{G}_a$ since the gradients of the consistency constraints at equilibrium of $\mathcal{G}_a$ are linearly independent. In the following proposition we state the result which is derived from \cite{di1989exact}.

\begin{proposition}(\textit{Theorem 4} in \cite{di1989exact}) \label{prop_MFCQ}
Let $(u^*,x^*) \in \mathcal{F}^{'}$ be an equilibrium of $\mathcal{G}_a$. Assume Mangasarian-Fromowitz constraint qualification (MFCQ) is satisfied at ($u^*,x^*$) and suppose ${\mathcal{R}}_i (u^{-i*}) \subset \mathring{\mathcal{D}}_i$. Then there exist a $\mu_i^*, i \in \mathcal{N}$ such that for all $\mu_i \in [\mu_i^*,\infty)$, $(u^*,x^*)$ is an equilibrium of the {exact penalty function game}.
\end{proposition}

     It can be seen that though in the paper by Di Pillo and Grippo, all the equality constraints are penalized, the result still holds even if we penalize a subset of the equality constraints as in our case. Let $(u^*,x^*)\in \mathcal{F}^{'}$ be an equilibrium of DTDG with consistent conjectures.     Suppose for $(u^{-i*},x^{-i*})$, the player $i$ plays the best response in the original game. In that case as we have done in the linear-quadratic case, let us denote $\mathbf{x}^i$ as the best response state trajectory conjectured by player $i$ in the original game. The relation of an equilibrium in $\mathcal{G}_c$ to the $\epsilon$-Nash equilibrium of $\mathcal{G}_o$ is stated as follows.
\begin{theorem}\label{Theorem_epsilon}
    Consider a DTDG with consistent conjectures given by  $\mathcal{G}_c$. Suppose $(u^*,x^*) \in \mathcal{F}^{'}$ is an equilibrium of $\mathcal{G}_c$. Suppose $x^{j(i)*} \in X^{j(i)}$ is the state trajectory conjectured by player $j(i) \in \mathcal{N}$ at the equilibrium of $\mathcal{G}_c$, where $j(i)$ is given by \eqref{next_labelled_player}. Suppose that ${\mathcal{R}}_i (u^{-i*}) \subset \mathring{\mathcal{D}}_i$, for all $i \in \mathcal{N}$. Let  $(\mathbf{u}^i,\mathbf{x}^i) \in {\mathcal{R}}_i(u^{-i*})$ be a best response state trajectory conjectured by player $i$ in $\mathcal{G}_o$ corresponding to $(u^{-i*},x^{-i*})$.
      Then $(u^{*},x^{*})$ is an $\epsilon$-Nash equilibrium of the original game with $\epsilon= \underset{i \in \mathcal{N}}{\max} \{\epsilon_i\}$, where $\epsilon_i$ satisfy ${\mu_i}|\mathbf{x}^i-x^{j(i)*}| \leq \epsilon_i$ and $\mu_i \in [\mu_i^*, \infty)$, where $\mu_i^*$ given by Proposition \ref{prop_MFCQ}.
    \end{theorem}
    \begin{proof}
        Suppose \{$(u^{i*},x^{i*}), i \in \mathcal{N}\}\in \mathcal{F}^{'}$ is an open-loop Nash equilibrium of $\mathcal{G}_c$. Since the equilibrium set of $\mathcal{G}_c$ is equivalent to $\mathcal{G}_a$, \{$(u^{i*},x^{i*}) \in U^i\times X^i, i \in \mathcal{N}$\} is also an equilibrium of $\mathcal{G}_a$.  Since MFCQ is satisfied at an equilibrium of $\mathcal{G}_a$, by Proposition \ref{prop_MFCQ}, the equilibrium set of $\mathcal{G}_a$ is equivalent to that of the exact penalty function game for all $\mu_i \in [\mu_i^*, \infty)$. Hence, for all $i \in \mathcal{N}$ and for some $\mu_i \in [\mu_i^*, \infty)$,
        \begin{multline*}{J^i(u^{i*},x^{i*};u^{-i*},x_1)} \leq  \Bigl[  {J^i(u^i,x^i;u^{-i*},x_1)+}\\{{\mu_i}|x^i-x^{j(i)*}|}\Bigr], \forall {(u^i,x^i) \in \Omega_i (u^{-i*};x_1) \cap {\mathring{\mathcal{D}_i}}}.
        \end{multline*}
                
         Let $(\mathbf{u}^i,\mathbf{x}^i) \in \Omega_i (u^{-i*};x_1) $ be a best response of player $i$ in $\mathcal{G}_o$ corresponding to the equilibrium actions $u^{-i*}$ of $\mathcal{G}_c$. Since the best response set ${\mathcal{R}}_i(u^{-i*}) \subset \mathring{\mathcal{D}}_i$, we can write  $(\mathbf{u}^i,\mathbf{x}^i) \in \mathring{\mathcal{D}}_i$. Since $(\mathbf{u}^i,\mathbf{x}^i) \in \Omega_i(u^{-i*};x_1) \cap {\mathring{\mathcal{D}}_i}$, 
        $${J^i(u^{i*},x^{i*};u^{-i*},x_1)} \leq  {J^i(\mathbf{u}^i,\mathbf{x}^i;u^{-i*},x_1)+{\mu_i}|\mathbf{x}^i-x^{j(i)*}|}.$$
        Since $(\mathbf{u}^i,\mathbf{x}^i)$ is a best response of the original game,
         \begin{multline*}
        {J^i(u^{i*},x^{i*};u^{-i*},x_1)} \leq \\ \underset{(u^i,x^i) \in \Omega_i (u^{-i*};x_1)}{\inf} \Bigl[ {J^i({u}^i,{x}^i;u^{-i*},x_1) \Bigr]+}  {{\mu_i}|\mathbf{x}^i-x^{j(i)*}|}.
        \end{multline*}
        
        Since ${\mu_i}|\mathbf{x}^i-x^{j(i)*}| \leq \epsilon_i, \forall i \in \mathcal{N}$ and 
        $\epsilon = \max \{\epsilon_i; i \in \mathcal{N}\} $, 
        \begin{multline*}
            {J^i(u^{i*},x^{i*};u^{-i*},x_1)} \leq  \underset{(u^i,x^i)}{\inf} \Bigl[ {J^i({u}^i,{x}^i;u^{-i*},x_1) \Bigr]+\epsilon_i},
            \end{multline*}
            where the $\inf$ is over $(u^i,x^i)\in \Omega_i (u^{-i*};x_1)$.
            That is $(u^{i*},x^{i*}; i \in \mathcal{N})$ is an $\epsilon$-Nash equilibrium of the original game.
    \end{proof} 
    This result gives conditions for an equilibrium of the game with consistent conjectures to be an $\epsilon$-Nash equilibrium of the original game.    A significant factor which determines the $\epsilon$-Nash equilibrium relation is the deviation of state conjectures at the equilibrium of $\mathcal{G}_c$ to the best response state conjecture of a player in $\mathcal{G}_o$ corresponding to the state conjectures at the equilibrium of $\mathcal{G}_c$.

\section{Conclusion} \label{conclusion}
The paper provides new results regarding the existence of open-loop Nash equilibria in DTDGs. A new approach for analysis of DTDGs using the state conjecture formulation is the pivotal step which leads to the new results. Using this formulation, a solution of an optimisation problem is related to an equilibrium for a certain class of DTDGs. The first result that we provide is for a class of games called quasi-potential DTDGs. Many games including linear-quadratic DTDGs under certain conditions comes under the class of quasi-potential DTDGs. Further, we modify the DTDG with an additional constraint of consistency of  the conjectured states for which the game has a shared constraint structure. Although the game is different from the original game, an equilibrium of the original game is also an equilibrium of the DTDG with consistent conjectures. Utilising the shared constraint structure, conditions for the existence of equilibria for a class of games called potential DTDGs are provided in this modified game with consistent conjectures.  A detailed section on the $\epsilon$-Nash equilibrium gives a relation of an equilibrium of the DTDG with consistent conjectures to the original game.

\ifCLASSOPTIONcaptionsoff
  \newpage
\fi



%

\bibliographystyle{IEEEtran}
\bibliography{IEEEabrv,sigproc}

\end{document}